%% file: main.tex
\documentclass[11pt]{amsart}

\usepackage{amsmath,amssymb,cite,mathrsfs,tikz-cd, amsfonts}
\usepackage{stmaryrd}
\usepackage[all]{xy}
\usepackage{xcolor} 

\usepackage{geometry}

\usepackage{verbatim}
\usepackage{mathtools}
\usepackage[bookmarks,colorlinks,breaklinks]{hyperref}  
\hypersetup{linkcolor=magenta,citecolor=orange,filecolor=blue,urlcolor=blue} 

\usepackage[nameinlink]{cleveref}

\usepackage{enumitem}

\newtheorem{theorem}{Theorem}[section]
\newtheorem{proposition}[theorem]{Proposition}
\newtheorem{lemma}[theorem]{Lemma}

\newtheorem{cor}[theorem]{Corollary}

\newtheorem{problem}[theorem]{Problem}
\newtheorem{conjecture}[theorem]{Conjecture}

\theoremstyle{definition}
\newtheorem{definition}[theorem]{Definition}

\theoremstyle{plain}
\numberwithin{equation}{theorem}

\theoremstyle{remark}
\newtheorem{remark}[theorem]{Remark}

\DeclareMathOperator{\Spec}{Spec} 
\DeclareMathOperator{\Proj}{Proj}

\DeclareMathOperator{\id}{id}

\DeclareMathOperator{\Pic}{Pic}

\DeclareMathOperator{\Div}{div}

\DeclareMathOperator{\DDiv}{Div}

\def\O{{\mathcal O}}

\newcommand{\bP}{{\mathbb P}}
\newcommand{\bZ}{{\mathbb Z}}
\newcommand{\bG}{{\mathbb G}}

\newcommand{\bC}{{\mathbb C}}

\newcommand{\bA}{{\mathbb A}}
\newcommand{\bQ}{{\mathbb Q}}

\newcommand{\s}{\mathcal{S}}

\newcommand{\Cu}{\mathcal{C}}

\usepackage{microtype}

\theoremstyle{remark}
\newtheorem{claim}{Claim}

\title{An Hilbert Irreducibility Theorem for integral points of Del Pezzo surfaces}
\author{Simone Coccia}

\address{
	Simone Coccia \\
	Department of Mathematics\\
	University of British Columbia\\
	Vancouver, BC V6T 1Z2\\
	Canada
}

\email{scoccia@math.ubc.ca}

\begin{document}
	
	\begin{abstract}
        We prove that the integral points are potentially Zariski dense in the complement of a reduced effective singular anticanonical divisor in a smooth del Pezzo surface, with the exception of $\bP^2$ minus three concurrent lines (for which potential density does not hold). This answers positively a question raised by Hassett and Tschinkel and, combined with previous results, completes the proof of the potential density of integral points for complements of anticanonical divisors in smooth del Pezzo surfaces. We then classify the complements which are simply connected and for these we prove that the set of integral points is potentially not thin, as predicted by a conjecture of Corvaja and Zannier. 
        \end{abstract}

	\maketitle
	
	\input{geometry}

	\input{integral_points}

\input{topology_new}
	\input{proof_nuova}
	\input{del_pezzo}

	\bibliography{mybibfile}
	\bibliographystyle{alpha}

\end{document}

%% file: geometry.tex
\section{Introduction}

\subsection{Notation}
We denote by $X$ a projective variety and by $D$ a reduced effective divisor of $X$, all defined over a number field $K$.
We say that the integral points on the pair $(X,D)$ are \emph{potentially dense} if there exists a finite extension $K'$ of $K$, a finite set of places $S$ of $K'$ containing the archimedean ones and a $\O_{K',S}$-integral model $(\mathcal{X},\mathcal{D})$ of $(X,D)$ such that the $\O_{K',S}$-integral points $(\mathcal{X} \setminus \mathcal{D})(\O_{K',S})$ are Zariski dense. If an $\O_{K,S}$-integral model for $(X,D)$ is fixed, for any closed subvariety $C \subseteq X$, we denote by $(C\setminus D)(\O_S)$ the set of points of $C(K)$ which are $S$-integral with respect to $D$. For a set $A \subset X$, we denote its Zariski closure by $\overline{A}^{{\rm Zar}}$.
\subsection{Potential density}

One central theme in Diophantine geometry is to find conditions under which the integral points of $(X,D)$ are potentially dense. It is expected that this should only depend on the geometry of the pair and roughly speaking that the \lq\lq more positive\rq\rq{} (resp. negative) the log-canonical divisor of the pair is, the \lq\lq more abundant\rq\rq{} (resp. scarce) the integral points are (potentially).
A precise web of conjectures has been advanced by Campana \cite{campana_fourier_2004}, who has defined so called \emph{special} pairs\footnote{We remark that Campana proposed conjectures in the more general context of orbifold pairs.} and conjectured that they should be precisely the ones admitting a potentially dense set of integral points. One important class of special pairs is the one with trivial log-canonical divisor, that is $K_X+D\sim 0$. In particular, one expects the following:
\begin{conjecture}[Campana, Hassett-Tschinkel]\label{log_k3}
    Let $(X,D)$ be a pair with $X$ a smooth projective variety and $D$ a reduced effective anticanonical divisor with at most normal crossings singularities. Then the integral points on $(X,D)$ are potentially dense.
\end{conjecture}
These pairs are particularly interesting since they sit in between the cases of positive and negative log-canonical divisor: while the integral points are expected to be dense, they should not be \lq\lq too abundant\rq\rq{}. In the case of surfaces, \Cref{log_k3} is already a major open problem for $D=\emptyset$, which amounts to potential density of rational points on K$3$ surfaces. More generally, when $X$ is a surface and $D$ is non-empty anticanonical we call such a pair a \emph{log-K$3$ surface}. 
\Cref{log_k3} has been solved in the case of $X=\bP^2$ and $D$ singular by Silverman \cite{silverman} and of $X=\bP^2$ and $D$ smooth by Beukers \cite{beukers}. Hassett and Tschinkel have proved it for $X$ a smooth del Pezzo surface and $D$ smooth, leaving open the case of $D$ singular. Our first result is to answer positively their expectation:
\begin{theorem}\label{soluzione_hassett}
    Let $(X,D)$ be a pair with $X$ a smooth del Pezzo surface and $D$ a reduced effective anticanonical divisor. Assume that the pair is not $(\bP^2,D)$ with $D$ the union of three concurrent lines. Then the integral points of $(X,D)$ are potentially dense.
\end{theorem}
\begin{remark}
Notice that \Cref{soluzione_hassett} includes also all cases of $D$ not normal crossings, except $\bP^2$ minus three concurrent lines. Indeed this surface is isomorphic to $\bA^1\times \bP^1\setminus \{0,1,\infty \}$ and so the integral points are never dense because of Siegel's Theorem applied to $\bP^1\setminus \{0,1,\infty \}$.
\end{remark}

\subsection{Hilbert Property}
It is natural to ask whether there is some property which is stronger than Zariski density and which could be potentially satisfied by the set of integral points on smooth del Pezzo surfaces minus an anticanonical divisor. One possible candidate is the (potential) Integral Hilbert Property, a geometric generalization to arbitrary varieties of the classical Hilbert's irreducibility theorem. More precisely, we say that an integral model $(\mathcal{X},\mathcal{D})$ for a pair $(X,D)$ has the Integral Hilbert Property, which we abbreviate with IHP, if the set of integral points is not thin (see \Cref{def_thin}). In \cite{CZ_hilbert} Corvaja and Zannier have asked the following question:
\begin{problem}[\hspace{1sp}\cite{CZ_hilbert}]\label{potente}
    Let $(X,D)$ be a pair with $X$ a smooth projective variety and $D$ a reduced effective divisor with at most normal crossings singularities, both defined over $K$. Assume that $X\setminus D$ is simply connected\footnote{In order for the IHP to hold, it is necessary to assume that $X\setminus D$ is simply connected (see \Cref{corollario_cw}).}. Let $(\mathcal{X},\mathcal{D})$ be an $\O_S$-integral model for $(X,D)$ with a Zariski dense set of integral points. Does $(\mathcal{X},\mathcal{D})$ have the IHP, possibly after an enlargement of $K$ and $S$?
\end{problem}
We remark that there is a more general version of \Cref{potente} asking whether the Weak Hilbert Property\footnote{This means that, for any finite collection of \emph{ramified} covers (see \Cref{def_cover}) of $X\setminus D$, there is always a Zariski dense set of points of $(X\setminus D)(\O_S)$ that do not lift to any cover of the collection.} holds for all pairs satisfying potential density of integral points (including the case of $X\setminus D$ \emph{not} simply connected). 
In its full generality, this is a very difficult question: already for $D$ empty (that is for rational points) and $X$ unirational, a positive answer would give a positive solution to the Inverse Galois Problem for finite groups (see \cite{thelene_sansuc}). \Cref{potente} is solved for curves, the main case being Hilbert's irreducibility theorem, which can be formulated in this language as the following:
\begin{theorem}[Hilbert]\label{hilbert}
	The affine line $\mathbb{A}^1_{\bZ}$ has the IHP.
\end{theorem}
Most of the known instances of \Cref{potente} have been obtained for $D=\emptyset$. In this case one simply talks about the Hilbert Property (HP), which amounts to $X(K)$ being not-thin. The HP is clearly a birational invariant; it holds for $\bP^n_K$ and so for all rational varieties. The first example of non-unirational variety with the HP has been given by Corvaja and Zannier in \cite{CZ_hilbert}: this is the Fermat quartic in $\bP^3$ of equation $x^4+y^4=z^4+w^4$. They developed a technique involving the use of two elliptic fibrations to produce \lq\lq many\rq\rq{} rational points on the surface, thereby proving the HP over $\bQ$.
In recent years several cases of \Cref{potente} and its more general version have appeared in \cite{coccia_ihp, corvaja_solo, corvaja_demeio_javanpeykar_lombardo_zannier_2022, demeio_ijnt, demeio_imrn, zannier_ferretti, mezzedimi_gvirtz, gvirtz_huang, javanpeykar, javanpeykar_2, streeter_campana, streeter_delpezzo, zannier_hilbert_algebraic_groups}.
See also \cite{soroker_fehm_petersen_1, soroker_fehm_petersen_2,demeio_streeter,loughran_salgado,luger} for related work.

Few instances of \Cref{potente} are known for integral points on varieties, such as $\bA^n$ (which reduces to \Cref{hilbert}) and the Weak IHP for $\bG_m^n$ (see \cite{corvaja_solo,zannier_ferretti,zannier_hilbert_algebraic_groups}). Notice that $\bA^n$ is a \emph{log-Fano} variety, that is, the opposite of the log-canonical divisor is ample: in particular, one expects that it should contain \lq\lq many\rq\rq{} integral points. Indeed all log-Fano varieties are special by \cite[Corollaire $6.10$]{campana_jimj_2011} (hence conjecturally satisfy potential density of integral points) and simply connected (by \cite{qi_zhang_rational_connectedness} and \cite[Cor. $4.18$]{debarre}): one then expects the IHP to hold potentially.
This is not the case for log-K$3$ varieties: while the integral points are expected to be potentially dense (by \Cref{log_k3}), these varieties are not always simply connected and so the IHP cannot hold. This is precisely what happens for log-K$3$ curves, that is elliptic curves and conics with two points at infinity, which have a potentially dense yet thin set of integral points. Notice that this is in accordance with the fact that the integral points should be less abundant (but still dense) compared to the log-Fano case.

For surfaces there are examples of simply connected log-K$3$, such as smooth cubic surfaces with a smooth anticanonical divisor at infinity. These surfaces have a potentially dense set of integral points (see \cite{beukers},\cite{Hats}) so they should have the IHP potentially. This has been proved by the author in \cite{coccia_ihp} adapting the double fibration technique of Corvaja and Zannier to conic fibrations and integral points.
In this paper we give a positive answer to \Cref{potente} for all smooth del Pezzo surfaces with a (possibly singular) anticanonical divisor at infinity:
\begin{theorem}\label{teorema_principale}
    Let $X$ be a smooth del Pezzo surface and $D$ a reduced effective anticanonical divisor of $X$. If $X\setminus D$ is simply connected, then $(X,D)$ has the IHP potentially.
\end{theorem}
The simply connected surfaces of \Cref{teorema_principale} are explicitly described in \Cref{caratterizzazione_complementi} (up to a $\overline{K}$-isomorphism).

\subsection{The conic fibrations method}
Throughout the paper by \emph{conic} we mean a smooth rational curve meeting the divisor at infinity in at most two points. Conic fibrations will be our main tool to prove \Cref{teorema_principale}, since they can be used to produce \lq\lq many\rq\rq{} integral points. Indeed the theory of Pell's equations ensures that, if a conic contains one integral point then, under some conditions, it contains infinitely many (see \Cref{pell}). Whenever a surface admits two independent conic fibrations, say $\lambda$ and $\mu$, one can take a fiber $C$ of $\lambda$ contanining infinitely many integral points and look at the fibers of $\mu$ passing through integral points of $C$. If the hypotheses of Pell's equation theory are met, then these fibers of $\mu$ will contain infinitely many integral points, since they contain one integral point. One may repeat the reasoning switching the roles of the two fibrations and taking any fiber of $\mu$ containing infinitely many integral points in place of $C$. If some technical hypotheses are met this allows to produce a non-thin set of integral points on the surface.
This technique was developed by Corvaja and Zannier \cite{CZ_hilbert} for rational points and elliptic fibrations to prove that $x^4+y^4=z^4+w^4$ has the Hilbert Property over $\bQ$, using two elliptic fibrations on this surface. The following result contains the technical hypotheses which allow to produce a non-thin set of integral points on certain rational double covers of $\bP^1\times\bP^1$:

\begin{theorem}\label{fondamentalissimo}
Let $(X,D)$ be a pair given by a smooth projective surface and a reduced effective divisor, both defined over a number field $K$. Let $S$ be a finite set of places of $K$ containing the archimedean places. Let $(\mathcal{X},\mathcal{D})$ be an $\O_S$-integral model for the pair, so that we have a well-defined notion of $S$-integral points with respect to $D$. Let $f \colon X \dashrightarrow \bP^1 \times \bP^1$ be a dominant rational $K$-map of degree $2$ whose locus of indeterminacy is contained in $D$. Let $\lambda \colon X \dashrightarrow \bP^1$ and $\mu \colon X \dashrightarrow \bP^1$ be the compositions of $f$ with the projections $\pi_1$ and $\pi_2$ of $\bP^1\times \bP^1$ on the first and second factor, respectively. We let $C_l \coloneqq \overline{\lambda^{-1}(l)}^{{\rm Zar}}$ and $C'_m \coloneqq \overline{\mu^{-1}(m)}^{{\rm Zar}}$ for  $l, m \in \bP^1$. Assume that, for all but finitely many $l,m \in \bP^1$, $C_l$ and $C'_m$ are smooth rational curves intersecting $D$ in two points. Assume the following are satisfied:
    \begin{enumerate}[label=\textbf{H\arabic*}]
        \item \label{ipotesi_ramificazione} the Zariski closure of the branch locus of $f$ contains at most one fiber of $\pi_1$ and, if it contains one such fiber, then $\lambda(D_{\mu})$ \emph{is not} a point, where $D_{\mu}$ is the union of the irreducible components of $D$ which are \emph{not} constant under $\mu$.
        \item \label{ipotesi_non-thin} There exists a non-thin subset $M \subset \bP^1(K)$ such that for $m \in M$ the set $(C'_{m}\setminus D)(\O_S)$ is infinite and for infinitely many $p\in (C'_{m}\setminus D)(\O_S)$ the set $(C_{\lambda(p)}\setminus D)(\O_S)$ is also infinite.
        \item \label{ipotesi_topologica} Let $E$ be the union of the curves on which both $\lambda$ and $\mu$ are constant. Then $X\setminus (D \cup E)$ is simply connected. 
    \end{enumerate} 
    Then $(X\setminus D)(\O_S)$ is not thin.
\end{theorem}
\Cref{fondamentalissimo} will be applied to obtain certain cases of \Cref{teorema_principale} where the divisor at infinity is reducible.
The proof of \Cref{fondamentalissimo} can be found in \Cref{section4} together with a motivation for its hypotheses. For now we limit ourselves to make a few remarks.
\begin{remark}
    For a concrete example where \Cref{fondamentalissimo} applies, take $X$ a smooth cubic surface in $\bP^3$, $D$ a hyperplane section of $X$, $\ell$ and $ \ell'$ two coplanar lines contained in $X$, all defined over a number field $K$. Let $\lambda \colon X \to \bP^1$ (resp. $\mu \colon X \to \bP^1$) be the conic fibration associated to the pencil of planes containing $\ell$ (resp. $\ell'$). Consider the natural $\O_K$-integral model obtained taking the closures of $X$ and $D$ in $\bP^3_{\O_K}$.  
    Suitably extending $K$ and $S$, one can show that the hypotheses of \Cref{fondamentalissimo} are satisfied for $f = \mu_{\ell} \times \mu_{\ell'} \colon X \to \bP^1\times \bP^1$, so that we obtain the potential IHP for $X\setminus D$. Clearly, \Cref{teorema_principale} includes this as a special case.
\end{remark}
\begin{remark}
    A surface $X$ satisfying the hypotheses of \Cref{fondamentalissimo} is necessarily rational. This is because $X$ must be rationally connected, and rationally connected surfaces are rational.
\end{remark}
\begin{remark}\label{ipotesi_semplificata}
    \Cref{punti_infiniti_giusto} ensures that, if $K$ has a complex place and $\# S \ge 2$, whenever a smooth $K$-rational curve with at most two points at infinity contains one $S$-integral point it actually contains infinitely many. In particular, for such a ring of integers $\O_{K,S}$, hypothesis \ref{ipotesi_non-thin} reduces to:
    \begin{itemize}
        \item there exists a non-thin set $M \subset \bP^1(K)$ such that for $m \in M$ the set $(C'_{m}\setminus D)(\O_S)$ non-empty.
    \end{itemize}
    In this paper we will apply \Cref{fondamentalissimo} to prove \Cref{teorema_principale}, which is a potential result, so we will extend $K$ and $S$ and use the above simplified form of hypothesis \ref{ipotesi_non-thin}. However in order to apply \Cref{fondamentalissimo} on other rings of integers, for instance to study $\bZ$-integral points, one needs the general form of \ref{ipotesi_non-thin}. This is because there are conics with a non-empty yet finite set of $\bZ$-points, such as $x^2+y^2=1$. In a forthcoming paper \cite{coccia_sarnak} the author will use \Cref{fondamentalissimo} to give sufficient criteria for an affine cubic surface to have a non-thin set of $\bZ$-integral points. Moreover in \Cref{esempio_non_potenziale} we have included an instance of \Cref{teorema_principale} where the IHP holds over $\bZ$. 
\end{remark}
\begin{remark}
    In some cases one can prove the IHP using a single conic fibration: 
    this happens if the surface has a fibration in conics with \emph{only one} point at infinity and the set of fibers containing at least one integral point is not thin (see \Cref{fibrazione_non_thin}).  
\end{remark}

\subsection{Irreducible divisor at infinity}
The most interesting case of \Cref{teorema_principale}, which is not covered by \Cref{fondamentalissimo} and \Cref{fibrazione_non_thin}, is the one of the Hirzebruch surface obtained by blowing up $\bP^2$ in one point, with divisor at infinity given by the strict transform of an irreducible cubic curve passing through the blown-up point.
Specifically, we will prove the following:

\begin{theorem}\label{irreducible}
    Let $D$ be an absolutely irreducible cubic $K$-curve in $\bP^2$ and $P$ a smooth $K$-point of $D$. Let $\sigma \colon X \to \bP^2$ be the blow-up at $P$ and let $\widehat{D}$ be the strict transform of $D$. Then the set of integral points of $X \setminus \widehat{D}$ is potentially not thin.
\end{theorem}

\Cref{irreducible} does not follow from \Cref{fondamentalissimo} since $X$ has only one fibration in conics with two points at infinity, namely the pencil $\lambda \colon X\to \bP^1$ of (strict transforms) of lines through the blown-up point $P$. Still the same strategy of \Cref{fondamentalissimo} applies: the idea is to employ the one-dimensional family of plane conics used by Beukers in \cite{beukers} to prove potential density for the complement of a smooth cubic curve in $\bP^2$. 

The main issue one encounters is that the conics with infinitely many integral points constructed by Beukers are parametrized by a \emph{thin} subset of $\bP^1$, while we need a non-thin set of parameters as in \ref{ipotesi_non-thin} of \Cref{fondamentalissimo}. The way out is to reinterpret Beukers' conics in terms of the cubic surface $\s \colon w^3=F$, where $F \in K[x,y,z]$ is a homogeneous cubic equation for $D$. To a line $\ell \subset \s$ one can associate the pencil of planes through $\ell$, which gives a fibration in conics $\mu_{\ell} \colon \s \to \bP^1$. Let $\rho \colon \s \to \bP^2$ be the map $[x:y:z:w] \mapsto [x:y:z]$. Then the Beukers' conics are the $\rho$-images of the fibers of $\mu_{\ell}$. 

In this reinterpretation the conics with infinitely many integral points constructed by Beukers are the $\rho$-images of the fibers of $\mu_{\ell}$ passing through points of $\ell$ which are integral with respect to the plane at infinity $H \coloneqq (w=0)$. The set parametrizing these conics is $\mu_{\ell}((\ell\setminus H)(\O_S))$, which is thin since $\ell$ is a bisection of $\mu_{\ell}$.
However, after suitable enlargements of $K$ and $S$, we can find a line $\ell'$ which is skew with $\ell$ (hence a section of $\mu_{\ell}$) and such that $(\ell'\setminus H)(\O_S)$ is non-thin (by \Cref{pell}). We can then consider the Beukers' conics parametrized by points of $(\ell'\setminus H)(\O_S)$: this allows to find the sought non-thin set of parameters. In conclusion, after dealing with some technical complications, one can apply the methods of the proof of \Cref{fondamentalissimo} and obtain \Cref{irreducible}, thereby completing the proof of \Cref{teorema_principale}.

We stress that our construction of a non-thin set of parameters requires to extend the base field by adding the third roots of unity. This is not the case for Beukers' parametrization, which in several cases is defined over $\bQ$, for instance for the complement in $\bP^2$ of $x^3+y^3+z^3=0$. However there are also many cases where one can produce a non-thin set of parameters over $\bQ$ using a more involved argument: one such example is sketched in \Cref{esempio_non_potenziale} and it exploits two conic fibrations $\mu_{\ell}$ and $\mu_{\ell'}$ associated to two \emph{coplanar} $\bQ$-lines $\ell, \ell' \subset \s$ (where we are using the notation of the paragraphs above). 
This method of using linear systems of planes associated to coplanar lines in cubic surfaces will be explored in a forthcoming paper \cite{coccia_sarnak}, where the author will give sufficient conditions for a cubic surface to have a non-thin set of $\bZ$-integral points.

\subsection{Acknowledgements} I would like to warmly thank my advisor Dragos Ghioca for his constant encouragement and support, both technical and moral, and for reading this work and providing several suggestions that improved the presentation. I would also like to thank Pietro Corvaja for having introduced me to this topic, as well as for many fruitful exchanges in these years. Moreover, I would like to thank Umberto Zannier for an email that prompted me to study the Hilbert Property for integral points of cubic surfaces with a singular divisor at infinity, which ultimately led me to the the results of this paper and my forthcoming work \cite{coccia_sarnak}. I would also like to thank Umberto Zannier for inviting me to the Cetraro meetings in the past summers, which have been fundamental in the development of this work. I would like to thank Ariyan Javanpeykar for his interest and useful comments. Finally, I would like to thank Francesco Ballini, Julian Demeio, Federico Scavia, and Francesco Zucconi for their interest and fruitful conversations.

%% file: integral_points.tex
\section{Preliminaries}

\subsection{Integral points}

We will briefly recall some basic definitions and results on integral points on varieties, using \cite{ascher_turchet,corvaja,CZnuovo,Hats} as our main sources. Other general references on Diophantine geometry are \cite{L, lang_nt3,  hindry_silverman, B-G}.

Let $v$ be a finite place of a number field $K$ and $\mathcal{O}_v$ be the valuation ring at $v$, i.e. $\mathcal{O}_v=\{x\in K \, \vert\, \lvert x \rvert_v \le 1\}$. We denote by $\mathfrak{m}_v=\{x\in K \, \vert \, \lvert x \rvert_v<1 \}$ its maximal ideal and by $K(v)=\mathcal{O}_v/\mathfrak{m}_v$ its residue field (which is finite).
If $S$ is a finite set of places containing the archimedean ones we define the ring of $S$-integers as
\begin{equation*}
    \mathcal{O}_S=\{x\in K \, \vert \, \lvert x\rvert_v \le 1, \forall v \notin S \}.
\end{equation*}
Let us give a first definition of integral points on a quasi-projective variety.
\begin{definition}
    Let $U$ be a quasi-projective variety defined over a number field $K$. An \emph{integral model} over $\O_S$ for $U$ is a flat scheme $\mathcal{U} \to \Spec \O_{S}$ whose generic fiber is isomorphic to $U$. An $S$-integral point of this model is a section of $\mathcal{U} \to \Spec \O_{S}$. We denote the set of $S$-integral points by $\mathcal{U}(\O_S)$.  By abuse of notation we will write $U(\O_S)$ in place of $\mathcal{U}(\O_S)$ whenever the choice of the integral model is clear from the context or when the sentence is true for any choice of the integral model.
\end{definition}
We will also use a second definition of $S$-integral points:
\begin{definition}
    A \emph{pair} is a couple $(X,D)$ where $X$ is a projective variety and $D$ a reduced effective divisor, called \emph{divisor at infinity}. An integral model $(\mathcal{X},\mathcal{D})\to \Spec \O_S$ for the pair consists of two proper $\O_S$-integral models $\mathcal{X}$ and $\mathcal{D}$ of $X$ and $D$, respectively, such that $\mathcal{D}$ is a reduced effective divisor of $\mathcal{X}$. An $S$-integral point of the couple $(X,D)$, also called an $S$-integral point of $X$ with respect to $D$, is a section of $\mathcal{X}\to \Spec \O_S$ avoiding $\mathcal{D}$. We denote the set of $S$-integral points by $(\mathcal{X}\setminus \mathcal{D})(\O_S)$. By abuse of notation we will write $(X\setminus D)(\O_S)$ in place of $(\mathcal{X}\setminus \mathcal{D})(\O_S)$ whenever the choice of the integral model is clear from the context or when the sentence is true for any choice of the integral model.
\end{definition}  
These two definitions of integral points are essentially equivalent. Specifically, let $U$ be a quasi-projective variety and let $X$ be a compactification of $U$ such that $D\coloneqq X\setminus U$ is a divisor. For any given model $(\mathcal{X},\mathcal{D})\to \Spec \O_S$ of the pair $(X,D)$, we obtain a corresponding integral model $\mathcal{U} \coloneqq \mathcal{X}\setminus \mathcal{D} \to \Spec \O_S$ for $U$: by definition, the sets $\mathcal{U}(\O_S)$ and $(X\setminus D)(\O_S)$ are equal. We stress that the notion of integrality depends on the model: the same rational point can be integral for one model and not integral for a different one. 

The projective space $\bP^n$ has a canonical $\O_S$-integral model, namely $\bP^n_{\O_S}$. This allows to associate an integral model to a projective embedding $X \hookrightarrow \bP^n$ simply by taking the closure of $X$ inside $\bP^n_{\O_S}$. Then, given a pair $(X,D)$, one canonically associate to any projective embedding $X \hookrightarrow \bP^n$ an integral model for $(X,D)$ induced by $\bP^n_{\O_S}$. In this case the non-canonicity of the model is due to the different possible choices of a projective embedding for $X$.

\begin{remark}\label{notation_clarification}
Let $(X,D)$ be a pair defined over $K$ together with a chosen $S$-integral model $(\mathcal{X},\mathcal{D})$. Let $C\subset X$ be a closed subvariety. By abuse of notation we will write $C(\O_S)$ for the set $C(K)\cap (X\setminus D)(\O_S)$. Equivalently, $C(\O_S)=(\mathcal{C} \setminus \mathcal{D})(\O_S)$, where $\mathcal{C}$ is the closure of $C$ in $\mathcal{X}$.
\end{remark}

\begin{definition}
    Let $U$ be a quasi-projective $K$-variety and $\mathcal{U} \to \Spec \O_{K,S}$ an integral model of $U$. We say that the integral points on $\mathcal{U}$ are potentially dense if there exists a finite extension $K'$ of $K$ and a finite set of places $S'$ of $K'$ containing all the places above $S$, such that the base change $\mathcal{U}' \to \Spec \O_{K',S'}$ has a Zariski dense set of integral points.
\end{definition}
\begin{remark}\label{density_model_dependence}
   We stress that while the notion of integrality for a given rational point depends on the choice of the model $\mathcal{U}$, the notion of potential density only depends on the geometry of $U$ over $\overline{K}$.
   Therefore when stating potential results we will not specify the choice of the model, and in the proofs we will choose conveniently the compactifications and the integral models that will best suit our needs.
\end{remark}
Let us define the crucial geometric invariant of curves which controls the distribution of integral points:
\begin{definition}
    Let $C$ be an absolutely irreducible quasi-projective curve. Let $\widetilde{C}$ be a smooth compactification of a normalization $C'$ of $C$ and let $\widetilde{D}\coloneqq \widetilde{C}\setminus C'$. We define the \emph{Euler characteristic} of $C$ as $\chi(C)\coloneqq \deg(K_{\widetilde{C}}+D)=2g(\widetilde{C})-2+\# \widetilde{D}$. Similarly, for a pair $(C,D)$ with $C$ an absolutely irreducible projective curve and $D$ a reduced effective divisor of $C$, pick a normalization $\pi\colon C' \to C$ of $C$ and let $D'\coloneqq \pi^{-1}(D)$: we define the Euler characteristic of the pair $(C,D)$ as $\chi(C-D)= 2g(C')-2+\# D'$.
\end{definition}
Whenever the Euler characteristic of a curve is positive the integral points are finite, as proven by Siegel and Faltings in their celebrated theorems:
\begin{theorem}[Siegel-Faltings]\label{siegel}
    Let $C$ be an absolutely irreducible quasi-projective curve defined over a number field $K$ and such that $\chi(C)>0$. Then, for any integral model of $C$, the integral points are always finite.
\end{theorem}
\begin{remark}\label{remark_serre}
    If $C$ is a $K$-irreducible curve which is geometrically reducible, then $C(K)$ is always finite: indeed an easy Galois theory argument shows that the rational points are a subset of the intersection points of the absolutely irreducible components of $C$, and these points are always finite in number.
\end{remark}
Conversely, whenever the Euler characteristic is zero or negative, we have the following:
\begin{theorem}\label{pell}
    Let $C$ be an absolutely irreducible projective curve over $K$ whose normalization has genus $0$ and let $D$ be a reduced $K$-divisor on $C$ given by the sum of (at most two) distinct smooth $\overline{K}$-points of $D$. Let $S$ be a set of places of $K$ containing the archimedean ones. Assume that $C(\O_S)$ contains at least one non-singular point (for a given choice of a model). We have that:
    \begin{itemize}
        \item if $D$ is empty or a single $K$-rational point then $(C\setminus D)(\O_S)$ is infinite and non-thin (see \Cref{def_thin} for a definition of thin);
        \item if $D=P_1+P_2$ is the sum of two distinct $K$-rational points and $\lvert S \rvert \ge 2$ then $(C\setminus D)(\O_S)$ is infinite;
        \item  if $D=P_1+P_2$ is the sum of two points which are quadratic conjugates in a quadratic extension $K'$ of $K$ and there exists at least one place $v \in S$ which splits in $K'$ (that is, $P_1,P_2 \in C(K_v)$) then $(C\setminus D)(\O_S)$ is infinite.
    \end{itemize}
\end{theorem}
\begin{proof}
    See \cite{bilu}, \cite[Corollary~5.3.3]{corvaja} or \cite[Corollary~5.4]{Hats}. The statement about the set being non-thin is essentially Hilbert's irreducibility theorem (\Cref{hilbert}).
\end{proof}
We will repeteadly use the following:
\begin{cor}\label{punti_infiniti_giusto}
    In the notation of the above Theorem, if $K$ has a complex place and $\# S \ge 2$, then $(C\setminus D)(\O_S)$ is either empty or infinite.
\end{cor}

\subsection{Thin sets and the Hilbert Property}
We will now briefly recall some definitions and results regarding thin sets and the Hilbert Property. For a detailed discussion on these topics, we refer the reader to \cite{SeMW, SeTGT, CZnuovo, corvaja}.

\begin{definition}\label{def_cover}
A $K$-\emph{cover} $\pi \colon Y\to X$ is a finite surjective $K$-morphism of finite degree between two $K$-integral varieties. A point $p \in X(K)$ \emph{lifts} to $Y$ if there exists $q \in Y(K)$ such that $p=\pi(q)$.
\end{definition}
\begin{definition}\label{def_thin}
    Let $X$ be an integral $K$-variety and let $\Omega \subseteq X(K)$. We say that $\Omega$ is \emph{thin} if there exists finitely many $K$-covers $\pi_i \colon Y_i \to X$ of degree $>1$ and a proper closed subvariety $Z \subset X$ such that $\Omega \subseteq Z(K) \cup \bigcup \pi_i(Y_i(K))$.  
\end{definition}

By definition non-Zariski dense sets of rational points are always thin.
\begin{definition}\label{definition_ihp}
    Let $\mathcal{X} \to \Spec \O_S$ be an integral model for an integral $K$-variety $X$. We say that $\mathcal{X}$ has the \emph{Integral Hilbert Property} (IHP) if $\mathcal{X}(\O_S) \subset X(K)$ is not thin.
\end{definition}
If $\mathcal{X}$ has the IHP then the $S$-integral points are Zariski dense, so that the IHP is a stronger property than Zariski density.
\begin{remark}
    \Cref{definition_ihp} is the same as saying that $\mathcal{X}$ has the Hilbert Property as an arithmetic scheme in the sense of \cite[Definition $1.2$]{luger}.
\end{remark}

\begin{remark}\label{HP_model_dependence}
    Similarly to \Cref{density_model_dependence}, the \emph{potential} Hilbert Property only depends on the geometry over $\overline{K}$ of the quasi-projective variety. In particular, when studying the IHP potentially, we will conveniently choose a projective compactification of the quasi-projective variety and an integral model. Moreover we will frequently enlarge the divisor at infinity $D$: if the IHP holds for a larger divisor at infinity $D'$ containing $D$, then it holds also for the original variety, since the integral points with respect to $D'$ are also integral with respect to $D$.
\end{remark}
\begin{remark}\label{reduction_finite}
    In order to prove the IHP for a given $K$-variety $X$ it is sufficient to prove that for all finite sets of \emph{finite $K$-morphisms} $\pi_i \colon Y_i \to X$ of degree $\ge 2$ with $Y_i$ normal and geometrically irreducible one has that $X(\O_S) \setminus \bigcup \pi_i(Y_i(K))$ is Zariski dense (see \cite[Remark $1.7$]{coccia_ihp}). In this case, whenever a rational point lifts, it actually lifts to an integral point:
\end{remark}
\begin{proposition}\label{lift_integrale}
    Let $\pi\colon Y\to X$ be a finite map between two quasi-projective varieties, all defined over $K$. Then, for any given $S$-integral models of $X$ and $Y$, there exists a finite enlargement $S'$ of $S$ such that, for any $x\in X(\O_S)$, if there exists $y\in Y(K)$ such that $\pi(y)=x$, then $y\in Y(\O_{S'})$.
\end{proposition}
\begin{proof}
    Let $\mathcal{X}$ and $\mathcal{Y}$ be the given $\O_S$-models of $X$ and $Y$, respectively. Up to taking a finite enlargement $S'$ of $S$ and denoting $\mathcal{X}'=\mathcal{X}\times \Spec \O_{S'}$ and $\mathcal{Y}'=\mathcal{Y}\times \Spec \O_{S'}$, we may assume that there exists a finite $\O_{S'}$-morphism $\Pi' \colon \mathcal{Y}' \to \mathcal{X}'$ such that the induced map on the generic fibers is $\pi$.
    We have the following commutative diagram
    \begin{equation*}
        \begin{tikzcd}
        \mathcal{Y}' \arrow{r}{\Pi'} \arrow{dr}  & \mathcal{X}' \arrow{d} \\
        \Spec K \arrow{r}\arrow{u}{y}  & \Spec \O_{S'} \arrow[bend right]{u}[swap]{x}
        \end{tikzcd}
    \end{equation*}
    Our goal is to extend $y$ to a section $\overline{y} \colon \Spec \O_{S'} \to \mathcal{Y}'$ making the diagram commute. For any point $\mathfrak{p} \in \Spec \O_{S'}$, we have the following commutative diagram
    \begin{equation*}
        \begin{tikzcd}
        \mathcal{Y}' \arrow{r}{\Pi'} \arrow{dr}  & \mathcal{X}' \arrow{d} \\
        \Spec K \arrow{d}\arrow{r}\arrow{u}{y}  & \Spec \O_{S'} \arrow[bend right]{u}[swap]{x} \\
        \Spec \O_{K,\mathfrak{p}} \arrow{ur} 
        \end{tikzcd}
    \end{equation*}
    where $\O_{K,\mathfrak{p}}$ is the localization of $\O_K$ at $\mathfrak{p}$.
    Since $\mathcal{Y}'\to \mathcal{X}'$ is finite (and in particular proper), we can apply the valuative criterion of properness to find a \emph{unique} lift $y_{\mathfrak{p}}\colon \Spec(\O_{K,\mathfrak{p}}) \to \mathcal{Y}'$ making the diagram commute. Patching together these lifts (which is possible due to uniqueness) we find the sought $\O_{S'}$-section $\overline{y} \colon \Spec \O_{S'} \to \mathcal{Y}'$.
\end{proof}

If a finite map $Y\to X$ is also unramified, then \emph{all} $S$-integral points on $X$ lift to finitely many twists of $Y\to X$. This is the content of Chevalley-Weil's theorem:
\begin{theorem}[Chevalley-Weil]\label{chevalleyweil}
    Let $\pi \colon Y\to X$ be a finite unramified morphism of quasi-projective varieties over a number field $K$ and let $S$ be a finite set of places in $K$ containing the archimedean ones. Then there exists a finite set of places $S'$ containing $S$ and finitely many $K$-varieties $Y_1,\dots, Y_n$ endowed with $K$-morphisms $\pi_i \colon Y_i \to X$ such that
	\begin{itemize}
	    \item $X(\O_S) \subseteq \bigcup_i \pi_i(Y_i(\O_{S'}))$;
	    \item there exists isomorphisms $\psi_i \colon Y_i \to Y$ defined over $\overline{K}$ such that $\pi \circ \psi_i=\pi_i$.
	\end{itemize}
\end{theorem}
\begin{proof}
    See Theorem $3.4$ of \cite{CZnuovo}.
\end{proof}
The following consequence of Chevalley-Weil was first observed by Corvaja and Zannier in \cite{CZ_hilbert}:
\begin{cor}\label{corollario_cw}
    Let $X$ be a quasi-projective variety and let $Y \to X$ be a non-trivial étale cover of $X$. Then $X(\O_S)$ is a thin set (for any integral model of $X$). In particular, if $X$ is not simply connected, then it does not have the IHP.
\end{cor}
In particular, the IHP is \emph{strictly} stronger than Zariski density of integral points, as shown by non-simply connected varieties with a potentially dense set of integral points (e.g. elliptic curves and $\bG_m$).

\section{Potential density for del Pezzo surfaces}
We will collect some definitions and facts on the geometry of smooth del Pezzo surfaces.
\begin{definition}
    A \emph{del Pezzo surface} is a smooth projective surface with ample anticanonical class. The degree of a del Pezzo surface is the self-intersection of its canonical class.
\end{definition}
\begin{definition}
    Let $n\le 8$ be a positive integer and let $P_1,...,P_n$ be distinct points in $\bP^2$. We say that these points are in \emph{general position} if no three of them are collinear, no six of them lie on a conic, and no eight of them lie on a cubic having a node at one of them. 
\end{definition}
It is easy to see that $\bP^2$, $\bP^1\times \bP^1$ and blow-ups of $\bP^2$ in $n \le 8$ points in general position are del Pezzo surfaces of degree respectively $9$, $8$ and $9-n$. Del Pezzo proved that, over an algebraically closed field, these are all possible isomorphism classes:
\begin{theorem}\label{classification_delpezzo}
    All del Pezzo surfaces have degree $d$ at most $9$. Over an algebraically closed field, every del Pezzo surface is isomorphic either to $\bP^1\times \bP^1$ or to the blow-up of $\bP^2$ in $9-d$ points in general position.
\end{theorem}
We will frequently employ the following characterization of the Picard group of blowups:
\begin{lemma}\label{picard}
    Let $X \to \bP^2$ be the blow-up of $\bP^2$ in $n$ points in general position. Let $E_1, \dots, E_n$ be the exceptional divisors above the blown-up points. Then $\Pic(X)\cong \bZ h \oplus \bZ e_1 \oplus \cdots \oplus \bZ e_n$, where $h$ is the class of a line in $\bP^2$, $e_i$ is the class of the exceptional divisor $E_i$ and the incidence relations are as follows $h^2=1, e_i^2=-1, e_i.h=0$ and $e_i.e_j=0$ for $i\neq j$. The class of the canonical divisor is $-3h+\sum_{i=1}^{n} e_i$.
\end{lemma}
\begin{proof}
    This follows immediately from \cite[Ch. V, Prop. $3.2$ and Prop. $3.3$]{Hart}.
\end{proof}
We now characterize the reduced effective anticanonical divisors in smooth del Pezzo surfaces. The reduced effective anticanonical divisors in $\bP^2$ are simply the cubic curves, that is, the zero loci of homogeneous cubic forms without repeated factors. In $\bP^1\times \bP^1$ they are the zero loci of biquadratic polynomials without repeated factors. Let us look at the blowups $\sigma\colon X\to \bP^2$ along finitely many points in general position.
For $D$ an effective reduced divisor in $\bP^2$, we let $\widehat{D}\coloneqq \widetilde{D}+\sum_{P \in D} (E_P.D-1)E_P$, where $\widetilde{D}$ is the strict transform of $D$ under $\sigma$ and the sum runs along the blown-up points lying on $D$. Notice that whenever $D$ is smooth at the blown-up points, then $\widehat{D}$ equals the strict transform. We have the following characterization:
\begin{proposition}\label{caratterizzazione_divisori}
    Let $\sigma \colon \DDiv(X) \to \DDiv(\bP^2)$ be the map defined by $A\mapsto \sigma(A)$. Let \, $\widehat{} \, \colon \DDiv(\bP^2) \to \DDiv(X)$ be the map defined by $D\mapsto \widehat{D}$. Then $\sigma$ and $\, \widehat \,$ induce a bijection between reduced effective anticanonical divisors of $X$ and cubic curves in $\bP^2$ passing through all the blown-up points and with at most double singularities in the blown-up points.
\end{proposition}
\begin{proof}
    The claim follows immediately from \Cref{picard} and the fact that, for a plane cubic curve $D$ and a blown-up point $P \in D$, the intersection product $E_P.D$ equals the multiplicity of $D$ at $P$.
\end{proof}

We can now prove \Cref{soluzione_hassett}:
\begin{proof}[Proof of \Cref{soluzione_hassett}]
    Let us first observe that it is necessary to exclude the case of $\bP^2$ minus three concurrent lines: indeed this is isomorphic to $\bA^1\times (\bP^1\setminus \{0,1,\infty \})$ and so the integral points are never dense because of Siegel's theorem applied to $\bP^1\setminus \{0,1,\infty \}$.
    We proceed to consider the various isomorphism classes of del Pezzo surfaces as listed in \Cref{classification_delpezzo}. If $X=\bP^2$, potential density has been proved in \cite{silverman} for $D$ singular normal crossings and in \cite{beukers} for $D$ smooth. If $D \subset \bP^2$ is not normal crossings, potential density follows from \cite[Theorem $2$]{levin_yasufuku}: we still provide a short proof, since the same argument appears in \Cref{finitely_generated}. If $D \subset \bP^2$ is the union of a smooth conic and a tangent line, looking at the linear system of lines through the tangency point, one sees that $\mathbb{P}^2\setminus D \cong \bA^1 \times \bG_m$, so that the integral points are potentially dense. If $D\subset \bP^2$ is a cuspidal cubic, let $L$ be the principal tangent at the node: then, again looking at the lines through the node, one finds that $\bP^2 \setminus (D \cup L) \cong \bA^1 \times \bG_m$, so that one can find a Zariski dense set of integral points after extending $K$ and $S$. In particular, these points are $S$-integral with respect to $D$, finishing the proof of the claim for $X=\bP^2$.
    
    Assume that $X$ is a blow-up $\sigma\colon X\to \bP^2$ along finitely many points in general position of $\bP^2$. Following \Cref{caratterizzazione_divisori}, we denote by $\widehat{D}$ the given reduced effective anticanonical divisor of $X$ and by $D$ its $\sigma$-image in $\bP^2$. Let us first assume that $D$ is not the union of three concurrent lines. Since $D$ is a reduced effective anticanonical divisor of $\bP^2$, integral points are potentially dense on $U\coloneqq \bP^2 \setminus D$. Let $D'$ be the divisor obtained adding to $\widehat{D}$ all the exceptional divisors which are not already contained in $\widehat{D}$. Then the pair $(X,D')$ is a projective compactification of $U$: by \Cref{density_model_dependence}, since integral points are potentially dense for $U$, they are also for $(X,D')$. Since $\widehat{D}\subset D'$ we have that the $D'$-integral points are also integral with respect to $\widehat{D}$, proving the claim.
    
    Assume now that $D$ is the union of three concurrent lines $L_1,L_2,L_3$ all meeting in a point $Q$. Due to \Cref{caratterizzazione_divisori}, we have that $Q$ is not a blow-up point of $\sigma$. Let $P$ be one of the blown-up points; we may assume $P \in L_1$ without loss of generality. Consider the divisor $D'$ obtained adding to $\widehat{D}$ all the exceptional divisors except the one above $P$. Again by \Cref{density_model_dependence} and the fact that $\widehat{D}\subseteq D'$, it is sufficient to prove potential density for $U\coloneqq X \setminus D'$. An alternative compactification for $U$ is given by the blowup $\nu \colon Y \to \bP^2$ in $P$, with the strict transform $\widetilde{D}$ of $L_1+L_2+L_3$ as the divisor at infinity. Let $E$ be the exceptional divisor of $Y$ lying above $P$. Then the linear system of lines through $P$ gives $Y$ the structure of a ruled surface $Y \to E$ over $E \cong \bP^1$. Since $P \in L_1$ we have that the strict transform $\widetilde{L}_1$ of $L_1$ is a fiber of the ruling. The complement $Y\setminus\widetilde{L}_1 $ is a ruled surface over $\bA^1$ such that $E,  \widetilde{L}_2, \widetilde{L}_3$ are three mutually non-intersecting sections. This gives an isomorphism between $X\setminus \widehat{D}$ and $\bA^1\times \bG_m$, whose integral points are potentially dense.
    
    It remains to deal with $X=\bP^1\times \bP^1$. We only need to address the case of $D$ singular anticanonical, since the smooth case has already been treated in \cite{Hats}.
    A reduced effective $(2,2)$ divisor in $\bP^1\times \bP^1$ is the zero locus of a biquadratic polynomial $F \in K[x_0,x_1,y_0,y_1]$ with no repeated factor. Up to automorphism (and enlarging $K$ and $S$), we may assume that $D$ has a singularity in $([x_0:x_1],[y_0:y_1])=([1:0],[1:0])$. This means that $F$ has no monomials of the type $x_0^2y_0^2$, $x_0^2y_0y_1$ and $x_0x_1y_0^2$. Let $L \coloneqq (\bP^1 \times \{[1:0]\})$ and $M\coloneqq (\{[1:0]\}\times \bP^1)$. The complement of $L \cup M$ in $\bP^1\times \bP^1$ can be identified with $\bA^2$ with coordinates $x\coloneqq x_0/x_1$ and $y \coloneqq y_0/y_1$. The restriction of $D$ in $\bA^2$ is defined by the zero locus of a quadratic polynomial $f \in K[x,y]$ obtained de-homogeneizing $F$. It is sufficient to prove potential density for $(\bP^1\times\bP^1)\setminus (D\cup L\cup M) \cong \bA^2 \setminus (f=0)$, which is the affine plane minus a (possibly reducible) conic. We compactify $\bA^2$ adding the line at infinity $H$ to get $\bA^2 \cup H=\bP^2$. In this compactification the affine conic $f=0$ will correspond to a projective conic $C$ in $\bP^2$. The curve $C+H$ has degree $3$ and so $\bP^2 \setminus (H+C)\cong \bA^2\setminus (f=0) \cong (\bP^1\times\bP^1)\setminus (D\cup L\cup M)$ admits potentially a dense set of integral points, provided $C+H$ is not given by three concurrent lines. This happens precisely when $f=0$ is given by two parallel lines in $\bA^2$, that is when $F$ has the following shape
    \begin{equation*}
	    F=(a x_1y_0+b x_0y_1 +c x_1y_1 )(a x_1y_0+b x_0y_1 +d x_1y_1),
    \end{equation*}
    where $a,b,c,d \in K$.
    If $a$ or $b$ vanished, then $F$ would contain a repeated factor and so $D$ would not be reduced. We can then assume $a, b \neq 0$, so that the family of polynomials $G_t=a x_1y_0+b x_0y_1 +t x_1y_1 $ defines a family of irreducible $(1,1)$ divisors. Observe that the curves $G_t=0$ intersect each other only in $([1:0],[1:0])$ with multiplicity $2$, and they intersect $F=0$ only in $([1:0],[1:0])$.
    We enlarge $K$ and $S$ so that the hyphoteses of \Cref{punti_infiniti_giusto} are met. The curve $Z\coloneqq \bP^1\times \{[0:1]\}$ intersects $F$ in two distinct points so, enlarging $S$ so that $(Z \setminus F)(\O_S)\neq \emptyset$ and applying \Cref{punti_infiniti_giusto}, we get that $(Z \setminus F)(\O_S)$ is infinite.
    Then the curves $G_t$ passing through the integral points of $Z$ contain infinitely many $S$-integral points, again by \Cref{punti_infiniti_giusto}. We have then found an infinite family of curves containing infinitely many $S$-integral points, thereby proving potential density.
\end{proof}
In the next section we will characterize the simply connected complements of anticanonical divisors in smooth del Pezzo surfaces: those are the ones which could potentially have the IHP, and indeed they do, as we prove in \Cref{teorema_principale}.

%% file: topology_new.tex
\section{Topology of complements of anticanonical divisors in Del Pezzo surfaces}
The aim of this section is to characterize the simply connected complements of reduced effective anticanonical divisors in smooth Del Pezzo surfaces. In the case of $\bP^2$ there is no such complement:
\begin{proposition}\label{p2notsimplyconnected}
    Let $D$ be a reduced effective anticanonical divisor of $\bP^2$. Then $\bP^2 \setminus D$ is not simply connected.
\end{proposition}
\begin{proof}
    A reduced effective anticanonical divisor of $\bP^2$ corresponds to a homogeneous degree $3$ form $F(x,y,z)$ with no repeated factors. The surface $Y \colon w^3=F(x,y,z)$ in $\bP^3$ is irreducible, smooth outside of $w=0$ and is equipped with a degree $3$ map to $\bP^2$ defined by $[x:y:z:w]\mapsto [x:y:z]$. This map restricts to a non-trivial étale cover $Y \setminus (w=0) \to \bP^2 \setminus D$, thereby proving $\bP^2 \setminus D$ is not simply connected.
\end{proof}
The same is true for $\bP^1\times \bP^1$:
\begin{proposition}\label{product_not_simply_connected}
    Let $D$ be a reduced effective anticanonical divisor of $\bP^1 \times \bP^1$. Then $(\bP^1 \times \bP^1)\setminus D$ is not simply connected.
\end{proposition}
\begin{proof}
    Let us pick an horizontal fiber $H$ and a vertical one $V$ such that neither $H$ nor $V$ are contained in $D$. Since $D$ is anticanonical it has class $(2,2)$ in $\Pic(\bP^1 \times \bP^1)\cong \bZ\times \bZ$. It follows that the divisor $D-2H-2V$ is principal, and so there exists a rational function $f \in \bC(\bP^1 \times \bP^1)$ such that $\Div(f)=D-2H-2V$. Then there exists a surface $Y$ with function field $\bC(\bP^1 \times \bP^1)(\sqrt{f})$ together with a non-trivial double cover $\pi \colon Y\to \bP^1 \times \bP^1$ ramified only along $D$. In particular, $\pi$ restricts to a non-trivial étale double cover of $(\bP^1 \times \bP^1)\setminus D$, proving the claim.
\end{proof}
All the other smooth Del Pezzo surfaces are isomorphic (over an algebraically closed field) to a blow-up $\sigma\colon X\to \bP^2$ along at most $8$ points in general position of $\bP^2$. 
Following \Cref{caratterizzazione_divisori}, from now on we will denote by $\widehat{D}$ a given reduced effective anticanonical divisor of $X$ and by $D$ its $\sigma$-image in $\bP^2$. The following result gives a complete characterization of the simply connected complements $X\setminus \widehat{D}$:
\begin{proposition}\label{caratterizzazione_complementi}
    Using the previous notation, we have that:
    \begin{enumerate}[label=(\roman*)]
        \item \label{caso_irriducibile} if $D$ is an irreducible cubic, then $X \setminus \widehat{D}$ is simply connected if and only if $\sigma$ blows up a smooth point of $D$;
        \item \label{caso_conica} if $D=L \cup C$, where $L$ is a line and $C$ a smooth conic, then $X \setminus \widehat{D}$ is simply connected if and only if $\sigma$ blows up a point of $C \setminus L$;
        \item \label{caso_tre_rette} if $D$ is the sum of three lines, then $X \setminus \widehat{D}$ is simply connected if and only if $\sigma$ blows up at least two smooth points of $D$ lying on two distinct lines.
    \end{enumerate}
\end{proposition}
Before proving \Cref{caratterizzazione_complementi}, we collect some auxiliary statements.
\begin{lemma}\label{quoziente}
	Let $X$ be a (connected) complex manifold, $Z \subset X$ a proper closed complex submanifold, $Y = X\setminus Z$ its complement, $p \in Y$ a point. Then the inclusion $i \colon Y \hookrightarrow X$ induces a surjective homomorphism $i_{*}\colon \pi_1(Y,p) \rightarrow \pi_1(X,p)$ between the corresponding fundamental groups.
\end{lemma}
\begin{proof}
	This is \cite[Lemma $2$, Section $3$]{CZ_advances}.
\end{proof} 
\begin{lemma}\label{finitely_generated}
    Using the previous notation, if $X \setminus \widehat{D}$ is not $\bP^2$ minus three concurrent lines, then $\pi_1(X \setminus \widehat{D})$ is abelian and finitely generated.
\end{lemma}
\begin{proof}
    Using \Cref{quoziente}, if we can prove the statement for a larger complement of $X$, then also $\pi_1(X \setminus \widehat{D})$ will be abelian and finitely generated. Let us first assume that $D$ is not the union of three concurrent lines. Removing the exceptional divisors, it is sufficient to prove the statement for $U\coloneqq \bP^2\setminus D$. There are six possibilities:
	\begin{enumerate}
	    \item $D$ is a smooth cubic curve;\label{enum_1}
	    \item $D$ is a nodal cubic curve;\label{enum_2}
	    \item $D$ is a cuspidal cubic curve;\label{enum_3}
	    \item $D$ is the union of a conic and a line in general position;\label{enum_4}
	    \item $D$ is the union of a conic and a tangent line;\label{enum_5}
	    \item $D$ is the union of three lines in general position.\label{enum_6}
	\end{enumerate}
    In case (\ref{enum_1}) and (\ref{enum_2}) it follows by \cite[Thm. $2$]{serre_revetement} that $\pi^1(U) \cong \bZ/3\bZ$. In case (\ref{enum_3}), if we remove from $U$ the principal tangent at the cusp we get a surface isomorphic to $\bA^1 \times \bG_m$, whose fundamental group is $\bZ$. In case (\ref{enum_4}) again by \cite[Thm. $2$]{serre_revetement} one gets $\pi^1(U) \cong \bZ$. In case (\ref{enum_5}), looking at the linear system of lines through the tangency point, we find that this complement is isomorphic to $\bA^1 \times \bG_m$. In case (\ref{enum_6}) we have $U \cong \bG_m^2$, whose fundamental group is $\bZ^2$.
    
    It remains to deal with the case of $D$ given by three concurrent lines $L_1,L_2,L_3$. From the hypotheses, we may assume that $X$ is blown up in at least one point, which is necessarily distinct from the point where the three lines meet due to \Cref{caratterizzazione_divisori}. Then, contracting all exceptional divisors except one, we may assume without loss of generality that $X$ is the blow-up of $\bP^2$ in a single point $P \in L_1 \setminus (L_2\cup L_3)$. We have already proved in \Cref{soluzione_hassett} that $X\setminus \widehat{D}$ is isomorphic to $\bA^1\times \bG_m$, which has fundamental group equal to $\bZ$. 
    \end{proof}
We can now prove \Cref{caratterizzazione_complementi}.
\begin{proof}[Proof of \Cref{caratterizzazione_complementi}]
    Whenever $D$ has a singular point in a blown-up point $P$, by \Cref{caratterizzazione_divisori} the divisor $\widehat{D}$ contains the exceptional divisor $E_P$ above $P$. In particular, the topology of $X \setminus \widehat{D}$ is the same as the surface obtained contracting $E_P$. Therefore from now on we will assume that $\sigma$ does not blow up any singular point of $D$.

    First we prove statement \labelcref{caso_irriducibile}. Let us assume that no smooth point of $D$ is blown-up. Then $\sigma$ does not blow-up any point, so that $X \setminus \widehat{D} \cong \bP^2 \setminus D$, which is not simply connected by \Cref{p2notsimplyconnected}. Let us now assume that $\sigma$ blows-up at least one smooth point of $D$, and we will prove that $X\setminus \widehat{D}$ is simply connected. It is sufficient to show that $X \setminus \widehat{D}$ admits no non-trivial finite cyclic topological cover. Indeed, if $X \setminus \widehat{D}$ has no non-trivial finite cyclic topological cover then $\pi_1(X \setminus \widehat{D})$ has no non-trivial finite cyclic quotient. Using \Cref{finitely_generated}, $\pi_1(X \setminus \widehat{D})$ would be a finitely generated abelian group with no non-trivial cyclic quotient and so it would be trivial.
    
    Let $Y \to X\setminus \widehat{D}$ be a cyclic topological cover of degree $n>1$. By a Theorem of Grauert and Remmert (see \cite[Chapter 6]{SeTGT}), $Y$ has a structure of algebraic variety such that the covering map is algebraic. Then the function field of $Y$ is obtained taking the $n$-th root of a rational function $f \in \mathbb{C}(X\setminus \widehat{D})^{*}= \mathbb{C}(X)^{*}$. The property that $Y \to X\setminus \widehat{D}$ is unramified means that all the divisors of zeros and poles of $f$ in $X\setminus \widehat{D}$ have multiplicity divisible by $n$. In other words, for each irreducible curve $C \not\subset \widehat{D}$, the multiplicity of $C$ in the divisor $\Div(f)$ is divisible by $n$. We will show that the multiplicity in $\Div(f)$ of $\widehat{D}$ is also divisible by $n$: since $X$ is simply connected, this implies that $f=g^n$ for some $g \in \mathbb{C}(X)^{*}$, and so the cover is trivial, proving the claim. We have that
	\begin{equation*}
		\Div(f) \equiv_n a\widehat{D},
	\end{equation*}
	where $a \in \bZ/n\bZ$. Using that $\widehat{D}.E=1$ we find 
	\begin{equation*}
		0=\Div(f).E \equiv_n a\widehat{D}.E=a,
	\end{equation*}
    proving the claim.

    Next, we deal with statement \labelcref{caso_conica}. If we assume that $\sigma$ blows up a point of $C \setminus L$, the strategy above of intersecting with suitable divisor classes allows to prove that $X\setminus\widehat{D}$ has no non-trivial finite cyclic cover, hence is simply connected. We now assume that all blown-up points of $\sigma$ are contained in $L\setminus C$ and we prove that $X\setminus \widehat{D}$ has a non-trivial étale cover, hence is not simply connected. We have $\Div(f)=C-2L$ for a suitable $f\in \bC(\bP^2)$. We let $g \coloneqq \sigma^*(f) \in \bC(X)$, so that $\Div(g)=\widehat{C}-2\widehat{L}-2\sum_i E_i$, where the $E_i$'s are the exceptional divisors above the blown-up points. The field $\bC(X)(\sqrt{g})$ is the function field of a double cover of $X$ ramified only on $\widehat{C}$, and so a non-trivial étale cover of $X \setminus \widehat{D}$.
    The proof of statement \labelcref{caso_tre_rette} is similar to \labelcref{caso_conica}.
    \end{proof}

According to \Cref{potente}, and due to \Cref{soluzione_hassett}, we expect that the complements of \Cref{caratterizzazione_complementi} satisfy (potentially) the IHP. In order to confirm this expectation we will adapt the double fibration method of Corvaja and Zannier to the context of integral points on conic fibrations.

%% file: proof_nuova.tex
\section{Double fibration method}\label{section4}

This chapter is devoted to the proof of \Cref{fondamentalissimo} and its application to \Cref{teorema_principale}. We will first describe the strategy for the proof of \Cref{fondamentalissimo}, which will motivate the hypotheses in its statement. 

\subsection{Strategy of our proof for \texorpdfstring{\Cref{fondamentalissimo}}{Theorem $1.7$}}
We stress that from now on we will use the notation of \Cref{fondamentalissimo} and the convention of \Cref{notation_clarification}. Suitably changing integral model, we may assume that $f$ is a morphism.
Due to \Cref{reduction_finite}, in order to prove that $(X\setminus D)(\O_S)$ is not thin we may argue by contradiction by showing that there \emph{cannot} exist finitely many absolutely irreducible normal projective $K$-varieties $Y_i$ equipped with finite $K$-morphisms $\pi_i \colon Y_i \to X$ of degree $>1$ such that the set $(X\setminus D)(\O_S) \setminus \bigcup_i(\pi_i(Y_i(K)))$ is contained in a curve $Z$. 

We will first prove that the set $\Lambda \coloneqq \{ l \in \bP^1(K) \, \vert \,\text{$C_l(\O_S)$ is infinite}\}$ is not thin (see \Cref{just_need_one}) by taking the conics $C_l$ passing through points $p \in C'_m(\O_S)$ for $m$ varying in $M$ (here we use \ref{ipotesi_ramificazione} and \ref{ipotesi_non-thin}). We then have that the two conic fibrations $\lambda$ and $\mu$ both admit \lq\lq many\rq\rq{} fibers containing infinitely many integral points. Our strategy is to find infinitely many such fibers for which the restrictions of the covers $\pi_i$ to these fibers lift only finitely many integral points: this gives the sought contradiction, since we assumed that the covers $\pi_i$ lift all integral points (outside of the curve $Z$).

\begin{remark}\label{thin_fiber_liftano_sempre}
    It is essential that the set of fibers containing infinitely many integral points is \emph{not} thin: this is because the rational points (in particular the integral points) contained in a thin set $T \subset \bP^1(K)$ of fibers can always be lifted to a cover. Indeed, by definition of thin set, there exists a finite map $\phi \colon \mathcal{C} \to \bP^1$ with no $K$-rational section such that $T \subset \phi(\mathcal{C}(K))$. Then by base change one gets a finite map $X \times_{\bP^1} \mathcal{C} \to X$ lifting all the rational points contained in the fibers above $T$. Notice that $X \times_{\bP^1} \mathcal{C} \to X$ is a non-trivial cover since it does not have a $K$-rational section (this comes from the fact that the generic fiber of $\lambda \colon X \to \bP^1$ is irreducible over $\overline{K(\bP^1)}$).
\end{remark}

One useful observation is that, given a conic $C$ with a divisor at infinity $D$ consisting of two points, a cover of $C$ can lift at most finitely many integral points of $C\setminus D$, unless it is totally ramified on $D$ and étale on $C\setminus D$ (this is \Cref{trick_siegel}, which is a consequence of Siegel's \Cref{siegel}).
In order to understand whenever this happens \lq\lq generically\rq\rq{} for the restrictions of a cover $\pi_i$ to the the fibers of $\lambda$, we need the following definition: 
\begin{definition}\label{lambda_unramified}
    In the notation of \Cref{fondamentalissimo}, we say that a cover $\pi_i$ is \emph{$\lambda$-unramified} if
    \begin{itemize}
        \item $\pi_i^{-1}(C_l)$ is absolutely irreducible for all but finitely many $l$;
        \item $\pi_i$ is totally ramified on the components of $D$ which are not constant for $\lambda$;
        \item the branch locus of $\pi_i$ is contained in the union of $D$ and finitely many fibers of $\lambda$.
    \end{itemize}
    Otherwise we say that $\pi_i$ is \emph{$\lambda$-ramified}.
\end{definition}
One can show that if a cover $\pi_i$ is $\lambda$-ramified, then there exists a thin set $T_i \subset \bP^1(K)$ such that for $l\notin T_i$ the restriction $\pi_i \colon \pi_i^{-1}(C_l) \to C_l$ lifts only finitely many integral points of $C_l$. If all covers were $\lambda$-ramified we would find a contradiction since $\Lambda$ is not thin, so we can assume that there is at least one $\lambda$-unramified cover.
If one assumes that the $\pi_i$'s lift all points of $(X\setminus D)(\O_S)\setminus Z(K)$ then the $\lambda$-unramified covers behave in the opposite way: that is, there exists a thin set $T$ such that for $l \in \Lambda \setminus T$ \emph{all} points of $C_l(\O_S)$ are lifted to at least one of the $\lambda$-unramified covers. This is accomplished using that $\pi_i \colon \pi_i^{-1}(C_l)\setminus \pi^{-1}(D) \to C_l\setminus D$ is an ètale cover of homogeneous spaces for tori, which allows to exploit the group structure of the associated tori (see \Cref{everything_lifts}).

If some of the $\pi_i$'s were both $\lambda$-unramified and $\mu$-unramified then our strategy could fail: indeed these covers could possibly lift all the integral points in the fibers of both fibration outside a thin set of fibers, and the remaining $\pi_i$'s could lift the integral points in the thin set of fibers (see \Cref{thin_fiber_liftano_sempre}).
This is where hypothesis \ref{ipotesi_topologica} comes into play: it ensures that a $\lambda$-unramified cover $\pi_i$ is necessarily $\mu$-ramified, since otherwise $\pi_i$ would induce a non-trivial ètale cover of the simply connected surface in \ref{ipotesi_topologica}. Then, by the same arguments used for $\lambda$-ramified covers, there must be a thin set $T'$ such that for $m \notin T'$ only finitely many integral points of $C'_m(\O_S)$ lift to the $\lambda$-unramified covers.

Using \ref{ipotesi_ramificazione} and \ref{ipotesi_non-thin} one can find an $m\in M\setminus T'$\footnote{Actually one can find infinitely many such $m$.} for which there are infinitely many points $p \in C'_m(\O_S)$ such that $\lambda(p) \notin T$ and the conic $C_l$ passing through $p$, namely $C_{\lambda(p)}$, contains infinitely many integral points. Since $\lambda(p) \notin T$ then all points of $C_{\lambda(p)}(\O_S)$, including $p$, would be lifted to $\lambda$-unramified covers. However $m\notin T'$ so the $\lambda$-unramified covers lift only finitely many points of $C'_m(\O_S)$, which is a contradiction.

We will now collect a series of preliminary results which will be used in the proof of \Cref{fondamentalissimo}.

\subsection{Preliminaries}

We begin with a useful consequence of Siegel's Theorem:
\begin{lemma}\label{trick_siegel}
    Let $(C,D)$ be a pair given by a smooth projective $K$-rational curve $C$ and a $K$-divisor $D$ given by the sum of two distinct $\overline{K}$-points.
    Let $V$ be a projective curve which is irreducible over $K$ and let $\pi \colon V \to C$ be a finite map. Let $\nu \colon V' \to V$ be a normalization of $V$ and let $D' \coloneqq (\pi \circ \nu)^{-1}(D)$. Assume (at least) one of the following holds:
    \begin{enumerate}
        \item $\# D' \ge 3$, that is, $\pi \circ \nu$ is not totally ramified on $D$; \label{three_points}
        \item $\pi \circ \nu$ has a branch point outside the support of $D$. \label{ramified}
    \end{enumerate}
    Then only finitely many points of $C(\O_S)$ lift to $V(K)$.
\end{lemma}
\begin{proof}
    Since $\nu$ is an isomorphism outside of a finite set, it is enough to prove that only finitely many points of $(C\setminus D)(\O_S)$ lift to $V'(K)$.
    We can also assume that $V$ is absolutely irreducible, otherwise $V(K)$ would be finite by \Cref{remark_serre}.
    By \Cref{lift_integrale} there exists an enlargement $S' \supseteq S$ such that, if a point of $C(\O_S)$ lifts to $V'(K)$, then it actually lifts to $(V'\setminus D')(\O_{S'})$. It is then sufficient to show that $(V'\setminus D')(\O_{S'})$ is finite. If $\# D' \ge 3$ this follows from Siegel's \Cref{siegel}. If $\# D' < 3$ then $\pi \circ \nu$ is totally ramified on the two points at infinity of $C$ and also ramifies outside of $D$. Applying Riemann-Hurwitz we find that the genus of $V'$ is positive, so that $(V'\setminus D')(\O_{S'})$ is again finite by Siegel's \Cref{siegel}.
\end{proof}
We now need a preliminary Lemma:
\begin{lemma}\label{lemma_rango_sottogruppi}
    Let $G$ be an abelian group, $B$ a finite rank subgroup of positive rank. Let $A_i$ be subgroups of $G$ and $t_i \in G$ elements such that $B \setminus \cup_i (t_i+A_i)$ is finite. Then this complement is actually empty.
\end{lemma}
\begin{proof}
    Define $B_i \coloneqq A_i \cap B$. If $B \cap (t_i+A_i) = \emptyset$ we can remove $t_i+A_i$ from the union $\cup_i (t_i+A_i)$ since it does not contribute to the difference $B \setminus \cup_i (t_i+A_i)$. Hence we assume that $B \cap (t_i+A_i) \neq \emptyset$ for all $i$. Then there exists $s_i \in A_i$ such that $t_i+s_i \in B$. It follows that $B \cap (t_i+A_i)=t_i+s_i+B_i$. Then one concludes applying \cite[Lemma $3.2$]{CZ_hilbert} with $G=B$, $H_i=B_i$ and $h_i=t_i+s_i$.
\end{proof}
The following result goes in the opposite direction of \Cref{trick_siegel}:
\begin{proposition}\label{everything_lifts}
    Let $\widetilde{U}$ be a smooth projective $K$-rational curve and $D \subset \widetilde{U}$ a $K$-divisor given by two points (over $\overline{K}$). Let $U \coloneqq \widetilde{U}\setminus D$ and suppose we have fixed an $S$-integral model for $U$. Suppose we are given finitely many smooth projective curves $\widetilde{V}_i$ equipped with a finite map $\pi_i \colon \widetilde{V}_i \to \widetilde{U}$ which is totally ramified above $D$ and ètale everywhere else. Let $V_i \coloneqq \widetilde{V}_i \setminus \pi_i^{-1}(D)$.
    Assume that $U(\O_S)$ is infinite and that for all but finitely many $p \in U(\O_S)$, $p$ lifts to (at least) one of the $V_i$'s. Then every $p \in U(\O_S)$ lifts to (at least) one of the $V_i$'s.
\end{proposition} 
\begin{proof}
Take $p \in U(\O_S)$. Since $U$ is a smooth $K$-rational curve with two points at infinity then it is a homogeneous space for a $K$-form $G$ of $\bG_m$. Let $\mathcal{G}$ be an $S$-integral model of $G$ which acts on the $S$-integral model of $U$ (such $\mathcal{G}$ always exists by \cite[Proposition $5.1$]{Hats}), so that $\mathcal{G}(\O_S)$ acts on $U(\O_S)$. From the discussion in Section $5$ of \cite{Hats}, using that $U(\O_S)$ is infinite it follows that $\mathcal{G}(\O_S)$ has positive rank. Looking at the generic fiber of $\mathcal{G}$ we can identify $\mathcal{G}(\O_S)$ with a subgroup of $G(K)$. Since $U$ is a homogeneous space for $G$, the point $p \in U(\O_S)$ gives a $K$-isomorphism between $U$ and $G$ as follows: the point $p$ is sent to the identity $\id \in G$ and a point $x \in U$ is sent to the one and only $g \in G$ for which $g.p=x$.

If $V_i$ does not contain any $K$-rational point we can ignore it. Clearly $V_i(K) \neq \emptyset$ for at least one $i$. Consider one such $V_i$: we have that $\pi_i^{-1}(D)$ is given by two points defined over $\overline{K}$ and $\widetilde{V}_i$ is a smooth rational curve due to Riemann-Hurwitz applied to $\pi_i$. Then $V_i$ is a homogeneous space for a $K$-form $H_i$ of $\bG_m$ and, since $V_i$ contains a $K$-rational point, we obtain an associated $K$-isomorphism between $V_i$ and $H_i$. Composing the $K$-isomorphisms $H_i \cong V_i$ and $U \cong G$ with the ètale map $\pi_i \colon V_i \to U$ we obtain an ètale $K$-morphism $\phi_i \colon H_i \to G$.

Consider the orbit $\mathcal{G}(\O_S).p \subseteq U(\O_S)$: since all but finitely many of the points of $U(\O_S)$ are lifted to $V_i(K)$ for at least one $i$, it follows that all but finitely many of the points in $\mathcal{G}(\O_S)$ are lifted to $H_i(K)$ for at least one $i$.
Since $\phi_i$ is an ètale morphism between $K$-forms of $\bG_m$ then it is the composition of an algebraic group $K$-morphism $f_i \colon H_i \to G$ and the translation by a rational point $t_i \in G(K)$. 
Applying \Cref{lemma_rango_sottogruppi} with $G= G(K)$, $B=\mathcal{G}(\O_S)$, $A_i=f_i(H_i(K))$ and $t_i \in G(K)$ as the coset representatives, we find that \emph{all} points in $\mathcal{G}(\O_S)$ are lifted to points of $H_i(K)$ for at least one $i$. In particular, the identity $\id \in G(\O_S)$ lifts to a rational point of $H_j(K)$ for some $j$. The $K$-isomorphism $G \cong U$ sends $\id$ to $p$, so it follows that $p$ lifts to $V_j(K)$.
\end{proof}

We will now follow the notation of \Cref{fondamentalissimo}.
If $C\subset X$ is a smooth projective curve which is not constant under $\lambda$, then the restriction $\lambda_{\vert C} \colon C \dashrightarrow \bP^1$ extends to a morphism on all of $C$.
In particular, the restriction $\lambda_{\vert C'_m}$ to a smooth fiber $C'_m$ of $\mu$ extends to a finite morphism $\lambda \colon C'_{m} \to \bP^1$ for all but finitely many $m$. Since the degree of $f$ is the number of intersections of two generic $C_l$ and $C'_m$, we have that $\deg(\lambda \colon C'_{m} \to \bP^1)=2$ for all but finitely many $m$. In particular, the branch locus of $\lambda_{|C'_m}$ is given by two points. This allows to reformulate \ref{ipotesi_ramificazione} as follows:
\begin{itemize}
    \item either the branch locus of $\lambda_{|C'_m}$ has no constant component while $m$ varies; or the branch locus of $\lambda_{|C'_m}$ contains \emph{precisely one} point which is constant as $m$ varies and $\lambda(C'_m \cap D)$ is given by two points for all but finitely many $m$.
\end{itemize}
The main consequence of \ref{ipotesi_ramificazione} is the following:
\begin{lemma}\label{solo_finiti_liftano}
    Let $T \subset \bP^1(K)$ be a thin set. If \ref{ipotesi_ramificazione} holds, then for all but finitely many $m \in \bP^1(K)$ the set $\lambda(C'_m(\O_S)) \cap T$ is finite.
\end{lemma}
\begin{proof}
    It suffices to prove that, given a $K$-irreducible cover $\phi \colon Z \to \bP^1$ of degree $>1$, for all but finitely many $m \in \bP^1(K)$ the set $\phi(Z(K)) \cap \lambda(C'_m(\O_S))$ is finite.
    We may assume that $Z$ is absolutely irreducible, otherwise $Z(K)$ would be finite. Further, we may assume $Z$ is smooth, since taking its normalization changes $\phi(Z(K))$ on a finite set.

    Assume that the branch locus of $\lambda_{|C'_m}$ has no constant component as $m$ varies. Then for all but finitely many $m$ the branch locus of $\lambda_{|C'_m}$ does not intersect the branch locus of $\phi$. It follows that the fiber product $V_m$ of $\phi$ and $\lambda_{|C'_m}$ is a smooth absolutely irreducible curve. The branch locus of $\phi' \colon V_m \to C'_m$ contains at least four points, since $\lambda_{|C'_m}$ has degree $2$ and is unramified above the branch points of $\phi$. Since $\#(C'_m \cap D)= 2$, the map $\phi'$ must be ramified outside of $C'_m \cap D$. Using \Cref{trick_siegel} we have that only finitely points of $C'_m(\O_S)$ lift to $V_m(K)$, so that the statement follows.

    It remains to deal with the case where the branch locus of $\lambda_{|C'_m}$ contains \emph{precisely one} point which is constant as $m$ varies and $\lambda(C'_m \cap D)$ is given by two points for all but finitely many $m$.
    It follows that, for all but finitely many $m$, the conic $C'_m$ is a smooth $K$-rational curve such that $\#(C'_m \cap D)=2$ and there exists at least one point $p \in \bP^1(K)$ that belongs to the branch locus of $\lambda_{\vert C'_m}$ and is not contained in the branch locus of $\phi$. Let $V_m$ be the fiber product of $\phi$ and $\lambda_{\vert C'_m}$ and let $V'_m$ be a normalization of $V_m$. We claim that $V'_m$ is absolutely irreducible. Assume by contradiction that there exists a proper irreducible component $W$ of $V'_m$. Since the map $V'_m \to Z$ is finite of degree $\deg \lambda_{\vert C'_m}=2$, then the restriction to $W$ would be a degree $1$ morphism $W \to Z$ between smooth projective curves, hence an isomorphism. In particular, there would be a map $f \colon Z \to C'_m$ such that $\phi=\lambda_{\vert C'_m}\circ f$. However this is impossible since $\phi$ is not ramified above $p$, while $\lambda_{\vert C'_m}\circ f$ is.

    Let $x \in Z$ be a ramification point of $\phi$ such that $\phi(x)$ is not a branch point of $\lambda_{\vert C'_m}$. Then there exists two distinct points $y_1,y_2 \in C'_m$ with $\lambda(y_1)=\lambda(y_2)=\phi(x)$. It follows that $(x,y_1),(x,y_2) \in V_m$ are smooth points of $V_m$ and so $(x,y_1),(x,y_2)$ can be identified with two points of the normalization $V'_m$. Since $\phi$ is ramified at $x$, the map $\phi'\colon V'_m \to C'_m$ is ramified at $(x,y_1)$ and $(x,y_2)$. By hypothesis $\{y_1,y_2\}\neq C'_m\cap D$ since $\lambda(C'_m\cap D)$ is not a point. Then $\phi'$ is ramified outside of $C'_m\cap D$ and the statement follows applying \Cref{trick_siegel}.
\end{proof}
Using \ref{ipotesi_non-thin} one can show that $\lambda$ has \lq\lq many\rq\rq{} fibers with infinitely many integral points: 
\begin{cor}\label{just_need_one}
    If \ref{ipotesi_ramificazione} and \ref{ipotesi_non-thin} hold\footnote{It is actually sufficient for $M$ to be infinite, and not necessarily not thin.}, then the set $\Lambda \coloneqq \{l \in \bP^1(K) \, \vert \, C_l(\O_S) \, \text{is infinite} \}$ is not thin.
\end{cor}
\begin{proof}
    We let $A_m \coloneqq \{p \in C'_m(\O_S) \, \vert \, \text{$C_{\lambda(p)}(\O_S)$ is infinite}\}$. By hypothesis, if $m \in M$ then $A_m$ is infinite. Since $\lambda(\bigcup_{m \in M}A_m) \subseteq \Lambda$, it is sufficient to prove that $\lambda(\bigcup_{m \in M}A_m)$ is not thin. Assume by contradiction that this set is thin. Using \Cref{solo_finiti_liftano} and the fact that $M$ is infinite, we can find $m \in M$ such that $\lambda(C'_m(\O_S))\cap \lambda(\bigcup_{m \in M}A_m)$ is finite. But $\lambda(A_m)\subseteq \lambda(C'_m(\O_S))\cap \lambda(\bigcup_{m \in M}A_m)$ and $A_m$ is infinite, which gives a contradiction.
\end{proof}

\subsection{Proof of \texorpdfstring{\Cref{fondamentalissimo}}{Theorem $1.7$}}
We can finally proceed to the proof.

\begin{proof}[Proof of \Cref{fondamentalissimo}]
We can always assume that $f$ is a morphism defined on all of $X$, since blowing-up closed points in the indeterminacy locus of $f$ and adding the exceptional locus to the divisor at infinity, we can find an $\O_S$-integral model $(\mathcal{X}',\mathcal{D'})$ for a pair $(X',D')$ such that $(\mathcal{X}'\setminus \mathcal{D}')\cong_{\O_S} (\mathcal{X}\setminus \mathcal{D})$, all the hypotheses of \Cref{fondamentalissimo} are satisfied for $(\mathcal{X}',\mathcal{D'})$ and $f$ extends to a morphism on all of $X'$. 
More precisely, let $I$ be the indeterminacy locus of $f$. For a closed point $P \in I$ we denote by $\mathcal{P}$ its Zariski closure in $\mathcal{X}$, which is an $\O_S$-scheme. The blowup $\sigma_{\mathcal{P}}\colon \mathcal{X}_{\mathcal{P}} \to \mathcal{X}$ along $\mathcal{P}$ satisfies $\mathcal{X}_{\mathcal{P}} \setminus \sigma_{\mathcal{P}}^{-1}(\mathcal{P}) \cong \mathcal{X} \setminus \mathcal{P}$ as $\O_S$-schemes. Since $I \subset D$, this implies that $\mathcal{X}_{\mathcal{P}} \setminus \sigma_{\mathcal{P}}^{-1}(\mathcal{D}) \cong \mathcal{X} \setminus \mathcal{D}$ and so the $\O_S$-integral points for these complements are the same: in particular, $(\mathcal{X}_{\mathcal{P}} \setminus \sigma_{\mathcal{P}}^{-1}(\mathcal{D}))(\O_S)$ is not thin if and only if $(\mathcal{X} \setminus \mathcal{D})(\O_S)$ is not thin. It is easy to check that all the hypotheses of \Cref{fondamentalissimo} are satisfied for the pair $(\mathcal{X}_{\mathcal{P}}, \sigma_{\mathcal{P}}^{-1}(\mathcal{D}))$. Iterating this procedure we reduce to the case where $f$ is a morphism.

Assume by contradiction that $(X\setminus D)(\O_S)$ is thin. By \Cref{reduction_finite} there exists finitely many absolutely irreducible normal projective $K$-varieties $Y_i$ equipped with finite $K$-maps $\pi_i \colon Y_i \to X$ of degree $>1$ such that $(X\setminus D)(\O_S) \setminus \bigcup_i(\pi_i(Y_i(K)))$ is contained in a curve $Z$. Under these assumptions the branch locus of the maps $\pi_i$ is a divisor due to Zariski's Theorem on the purity of the branch locus. 
For one cover $\pi_i \colon Y_i \to X$ we denote the fiber above the parameter $l\in \bP^1(\overline K)$ by $Y_i(l)\coloneqq (\lambda \circ \pi_i)^{-1}(l) = \pi_i^{-1}(C_l)$. Using that $Y_i$ is a normal surface, Bertini's Theorem \cite[Remark $10.9.2$]{Hart} ensures that the curves $Y_i(l)$ are smooth for all but finitely many $l$.
Our first goal is to prove the following:
\begin{claim}\label{claim1}
    There exists a thin set $T \subset \bP^1(K)$ such that if $l \notin T$ then only finitely many points of $C_l(\O_S)$ lift to the \emph{$\lambda$-ramified} covers (see \Cref{lambda_unramified}). 
\end{claim}
We will show that for any $\lambda$-ramified cover $\pi_i$ there exists a thin set $T_i$ such that if $l \notin T_i$ then only finitely many points of $C_l(\O_S)$ lift to $Y_i(K)$. \Cref{claim1} then follows taking $T \coloneqq \cup_i T_i$.
We first treat the case of a $\lambda$-ramified cover for which $Y_i(l)$ is reducible over $\overline{K}$ for all but finitely many $l$.
\begin{lemma}\label{generically_reducible}
	If the generic fiber of $\lambda \circ \pi_i$ is geometrically reducible, there exists a cover defined over $K$ of $K$-irreducible curves $\phi \colon \Cu \to \mathbb{P}^1$ such that the pullback of $Y_i$ over the base change $\Cu \times_{\mathbb{P}^1}X$ is reducible over $\overline K$.
\end{lemma}
\begin{proof}
	See \cite[Lemma $3.9$]{coccia_ihp}.
\end{proof}
Using this Lemma and the fact that $Y_i$ is geometrically irreducible, we have $\deg(\phi \colon \Cu \to \mathbb{P}^1)>1$. Then the components of $Y_i(l)$ may be defined over $K$ only for $l \in \phi(\Cu(K))$, which is a thin set in $\mathbb{P}^1(K)$. That is, for $l \notin \phi(\Cu(K))$, the fiber $Y_i(l)$ is irreducible over $K$ but reducible over $\overline{K}$, and this implies that $Y_i(l)$ has only finitely many rational points by \Cref{remark_serre}. We may then take $T_i \coloneqq \phi(\Cu(K))$.

Let us consider the case of a $\lambda$-ramified cover $\pi_i$ for which $Y_i(l)$ is geometrically irreducible for all but finitely many $l$. 
We initially define $T_i$ to be the set of $l \in \bP^1(\overline{K})$ such that $Y_i(l)$ is either reducible over $\overline{K}$ or passes through a singular point of $Y_i$. Notice that $T_i$ is finite since $Y_i$ is normal and so has a finite singular locus. We will progressively enlarge $T_i$ as needed.

Let us first assume that the branch locus of $\pi_i$ contains an irreducible curve $B$ which is not contained in $D$ and is not constant for $\lambda$. Then the generic $C_l$ intersects transversally $B$ in at least one point which is smooth for both $C_l$ and $B$ and is not contained in the divisor at infinity $D$. We add the finitely many exceptions in $T_i$. We now use the following:
\begin{lemma}\label{restriction}
	Let $\pi \colon Y \to X$ be a finite map from a normal surface $Y$ to a smooth surface $X$. Let $B$ be one irreducible component of the branch divisor. Let $C$ be a curve in $X$ intersecting $B$ transversally at some point $p$ which is smooth for both $C$ and $B$ and such that $\pi^{-1}(p)$ contains no singular point of $Y$. Suppose that $\pi^{-1}(C)$ is absolutely irreducible. Then the restriction of $\pi$ to $\pi^{-1}(C)$ defines a cover $\pi \colon \pi^{-1}(C) \to C$ ramified above $p$.
\end{lemma}
\begin{proof}
    See \cite[Lemma $3.10$]{coccia_ihp}.
\end{proof}
It follows that for $l \notin T_i$ the restriction $\pi_i \colon Y_i(l) \to C_l$ is a cover of absolutely irreducible curves ramified above an affine point of $C_l$. Applying \Cref{trick_siegel} we have that only finitely many points of $C_l(\O_S)$ can lift to $Y_i(K)$ for $l \notin T_i$.
If there is no such curve $B$ in the branch locus of $\pi_i$, then $\pi_i$ is \emph{not} totally ramified on at least one of the components of $D$ which are not constant under $\lambda$. It follows that, for all but finitely many $l$'s which we add in $T_i$, the restriction $\pi_i \colon Y_i(l) \to C_l$ satisfies $\#((\pi_i)^{-1}(D\cap C_l)) \ge 3$. Again by \Cref{trick_siegel} we have that only finitely many points of $C_l(\O_S)$ lift to $Y_i(K)$.

\Cref{claim1} follows taking $T=\cup_i T_i$ for $i$ varying along all $\lambda$-ramified covers.
If all covers were $\lambda$-ramified we would find a contradiction: indeed by \Cref{just_need_one} the set $\Lambda \coloneqq \{l \in \bP^1(K) \, \vert \, C_l(\O_S) \, \text{is infinite} \}$ is not thin and so we can always find $l \in \Lambda \setminus T$ such that only finitely many points of $C_l(\O_S)$ lift to $\lambda$-ramified covers. Hence from now on we assume that there exists at least one $\lambda$-unramified cover. We will prove the following:
\begin{claim}\label{claim2}
    Up to further enlarging $T$ with a thin set, if $l \in \Lambda \setminus T$ then \emph{all} the points in $C_l(\O_S)$ are lifted to $\lambda$-unramified covers, that is for any $p \in C_l(\O_S)$ there exists a $\lambda$-unramified cover $\pi_i \colon Y_i \to X$ and $y \in Y_i(K)$ such that $\pi_i(y)=p$.
\end{claim}

We add to $T$ the finitely many $l$'s for which $C_l$ is contained in the branch locus of at least one $\lambda$-unramified cover. Then, up to further enlarging $T$ with finitely many parameters, we may assume that, for $l \notin T$ and for all $\lambda$-unramified covers $\pi_i$, both $C_l$ and $Y_i(l)$ are smooth and absolutely irreducible and the restriction $\pi_i \colon Y_i(l) \to C_l$ is étale outside $D \cap C_l$ and totally ramified on $D \cap C_l$. 
Recall that we are assuming that $(X\setminus D)(\O_S) \setminus \bigcup_i(\pi_i(Y_i(K)))$ is contained in a curve $Z$. We add to $T$ the finitely many values corresponding to irreducible components of $Z$ which are constant for $\lambda$. Thus for $l \notin T$ we have that all but finitely many points of $C_l(\O_S)$ lift to the covers $\pi_i$: however, since for $l \notin T$ only finitely many points of $C_l(\O_S)$ lift to $\lambda$-ramified covers, it follows that for $l \notin T$ all but finitely many of the points of $C_l(\O_S)$ lift to the $\lambda$-unramified covers.
Applying \Cref{everything_lifts} to the maps $\pi_i \colon Y_i(l) \to C_l$ for $\pi_i$ varying along all $\lambda$-unramified covers, we find that if $l \notin T$ and $C_l(\O_S)$ is infinite, then \emph{all} points of $C_l(\O_S)$ lift to the $\lambda$-unramified covers, thereby proving \Cref{claim2}.

Observe that $\lambda$-unramified covers are $\mu$-ramified. Indeed, if a cover was $\lambda$-unramified and $\mu$-unramified, then its branch locus would be contained in $D \cup E$, where $E$ is the union of the curves which are constant for both $\lambda$ and $\mu$. However, by hypothesis \ref{ipotesi_topologica}, this would force this cover to be trivial, contradicting the assumption.
We may then repeat the argument used for $\lambda$-ramified covers to find a thin set $T'$ such that for $m \notin T'$ only finitely many points in $C'_m(\O_S)$ lift to rational points of $\lambda$-unramified covers.

Using hypothesis \ref{ipotesi_non-thin} and \Cref{solo_finiti_liftano} we can pick an $m \notin T'$ such that $C'_m(\O_S)$ is infinite and there are infinitely many points $p \in C'_m(\O_S)$ such that $C_{\lambda(p)}(\O_S)$ is infinite and $\lambda(p) \notin T$. By the previous discussion all the points in $C_{\lambda(p)}(\O_S)$, and in particular the point $p$, would be lifted to (at least) one $\lambda$-unramified cover. But then there would be infinitely many points in $C'_m(\O_S)$ which are lifted to $\lambda$-unramified covers, which is a contradiction since $m \notin T'$.
\end{proof}

\subsection{Applications: reducible divisor at infinity}
We will now prove the potential IHP for the simply connected complements appearing in \Cref{caratterizzazione_complementi} and for which $D$ is reducible. We first need the following:
\begin{proposition}\label{fibrazione_non_thin}
    Let $(X,D)$ be a pair given by a smooth projective surface and a reduced effective divisor, both defined over a number field $K$. Let $S$ be a finite set of places of $K$ containing the archimedean places. Let $(\mathcal{X},\mathcal{D})$ be an $\O_S$-integral model for the pair, so that we have a well-defined notion of $S$-integral points with respect to $D$. Let $\lambda \colon X \dashrightarrow \bP^1$ be a $K$-pencil such that all but finitely many fibers are smooth rational curves intersecting $D$ only in one point. Let $C_{l} \coloneqq \overline{\lambda^{-1}(l)}^{{\rm Zar}}$. Assume that there exists a non-thin set $\Lambda \subset \bP^1(K)$ such that $C_l(\O_S)\neq \emptyset$ for $l \in \Lambda$. Then $(X\setminus D)(\O_S)$ is not thin.
\end{proposition}
\begin{proof}
    We assume by contradiction that there exists finitely many absolutely irreducible normal projective $K$-varieties $Y_i$ equipped with finite $K$-maps $\pi_i \colon Y_i \to X$ of degree $>1$ such that $(X\setminus D)(\O_S) \setminus \bigcup_i(\pi_i(Y_i(K)))$ is contained in a curve. Observe that for $l \in \Lambda$ the set $C_l(\O_S)$ is not-thin by \Cref{pell}. Following the same strategy as in the proof of \Cref{fondamentalissimo}, we can find $l\in \Lambda \setminus T$ such that the restrictions $\pi_i \colon \pi_i^{-1}(C_l) \to C_l$ are non-trivial covers lifting all but finitely many of the points in $C_l(\O_S)$. This is impossible since $C_l(\O_S)$ is not thin, thereby proving the claim.
\end{proof} 
\begin{theorem}\label{riducibile}
    Let $\sigma\colon X\to \bP^2$ be a blow-up along at most $8$ points in general position of $\bP^2$. Let $\widehat{D}$ be a reduced effective anticanonical divisor of $X$ and let $D\coloneqq \sigma(\widehat{D})$. Assume that $D$ is reducible over $\overline{K}$. Then $(X\setminus \widehat{D})$ has the IHP potentially.  
\end{theorem}
\begin{proof}
    Using \Cref{caratterizzazione_complementi} and \Cref{HP_model_dependence}, after contracting some exceptional divisors and enlarging $K$ and $S$, we may reduce to the following cases:
    \begin{enumerate}
        \item \label{conica_retta} $D$ is the sum of a smooth conic $C$ and a line $L$ and $\sigma$ is the blowup of $\bP^2$ at a point $P\in C\setminus L$;
        \item \label{tre_rette} $D$ is the sum of three lines $L_1,L_2,L_3$ and $\sigma$ is the blowup of $\bP^2$ in two points $P_1$ and $P_2$ lying on $L_1$ and $L_2$ respectively and distinct from the points of intersection of the three lines.
    \end{enumerate} 
    We first deal with the cases where $D$ has a singularity which is not normal crossings, which can be treated with \Cref{fibrazione_non_thin}.
    \begin{itemize}
    \item Assume we are in Case (\labelcref{conica_retta}) and $L$ is tangent to $C$ in $Q$. Let $L(x,y,z)=0$ and $C(x,y,z)=0$ be equations for $L$ and $C$, and consider the one-dimensional linear system of conics $G_l\coloneqq (aL^2+bC=0)$ for $l=[a:b] \in \bP^1$. This is the pencil of plane conics which intersect $C$ only in $Q$ with multiplicity $4$, and it defines a morphism $\lambda\colon X \setminus\{Q\}\to \bP^1$. Observe that all but finitely many $G_l$ intersect $\widehat{D}$ only in one point. The strict transform $F$ of the line through $P$ and $Q$ intersects $\widehat{D}$ only in $Q$, so enlarging $K$ and $S$ we may assume that $F(\O_S)$ is not thin (by \Cref{pell}). Observe that $F$ is a section of $\lambda$ and so $\Lambda \coloneqq \lambda(F(\O_S)$ is not thin. Applying \Cref{fibrazione_non_thin} we prove our claim.

    \item Assume we are in Case (\labelcref{tre_rette}) and the three lines meet in a common point $Q$. Consider the pencil of lines $G_l$ through $Q$, which defines a morphism $\lambda\colon X \setminus\{Q\}\to \bP^1$. The generic $G_l$ intersects $\widehat{D}$ only in $Q$. The strict transform $F$ of the line through $P_1$ and $P_2$ intersects $\widehat{D}$ in a single point, so that $F(\O_S)$ is potentially non-thin. We can take $\Lambda \coloneqq \lambda(F(\O_S))$ and conclude applying \Cref{fibrazione_non_thin}.
    \end{itemize}
    We will deal with the remaining cases using \Cref{fondamentalissimo}.
    First of all, we enlarge $K$ and $S$ so that the hypotheses of \Cref{punti_infiniti_giusto} are met. In particular, we will check the simplified form of \ref{ipotesi_non-thin} appearing in \Cref{ipotesi_semplificata}.
    \begin{itemize}
        \item  Assume we are in Case (\labelcref{conica_retta}) and $L$ and $C$ meet in two distinct points $Q_1$ and $Q_2$. We want to find two pencil of conics satisfying the hypotheses of \Cref{fondamentalissimo}. The first is the fibration $\mu \colon X \to \bP^1$ given by the strict transforms of the lines through $P$. The second pencil is given by the conics through $Q_1$ and $Q_2$ and with the same tangent as $C$ in both $Q_1$ and $Q_2$. Taking the strict transforms for $\sigma$, this family gives a map $\lambda \colon X\setminus \{Q_1,Q_2\} \to \bP^1$. The only fiber of $\lambda$ which is tangent to all lines through $P$ is the double line $L$. However $C$ is a component of $D$ which is horizontal for $\mu$, so that \ref{ipotesi_ramificazione} is satisfied.
        The exceptional divisor $E$ above $P$ intersects the divisor at infinity $\widehat{D}$ precisely in one point so, up to enlarging $K$ and $S$, $E(\O_S)$ is non-thin by \Cref{pell}. The exceptional divisor $E$ is a section of $\mu$ and so \ref{ipotesi_non-thin} is satisfied picking $M \coloneqq \mu(E(\O_S))$.
        There is no curve which is constant for both $\lambda$ and $\mu$, so also \ref{ipotesi_topologica} is satisfied. We may then apply \Cref{fondamentalissimo} and prove the IHP.
        
        \item Assume we are in Case (\labelcref{tre_rette}) and the three lines are in general position. Let $Q_{i,j} \coloneqq L_i \cap L_j$. The pencil of lines through $Q_{1,2}$ gives a map $\mu \colon X \setminus \{ Q_{1,2} \} \to \bP^1$. The strict transforms under $\sigma$ of the conics through $P_1,P_2, Q_{1,3}, Q_{2,3}$ give a map $\lambda \colon X\setminus \{ Q_{1,3}, Q_{2,3} \} \to \bP^1$. The only fiber of $\lambda$ which is tangent to all lines through $P$ is $L_1 \cup L_2$. However $L_3$ is a component of $D$ which is horizontal for $\mu$, so that \ref{ipotesi_ramificazione} is satisfied.
        The strict transform $F$ of the line through $P_1$ and $P_2$ meets $\widehat{D}$ in one point so, after enlarging $K$ and $S$, we may assume that $F(\O_S)$ is not thin.
        The curve $F$ is a section of $\mu$ and so \ref{ipotesi_non-thin} holds taking $M \coloneqq \mu(F(\O_S))$.  
        There is no curve which is constant for both $\lambda$ and $\mu$, so also \ref{ipotesi_topologica} is satisfied. 
        We may then apply \Cref{fondamentalissimo} and prove the IHP.
    \end{itemize}
\end{proof}

%% file: del_pezzo.tex
\section{Irreducible divisor at infinity}\label{section5}

This section is devoted to the proof of \Cref{irreducible}, which allows us to complete the proof of \Cref{teorema_principale}:
\begin{proof}[Proof of \Cref{teorema_principale}]
    Due to \Cref{product_not_simply_connected} and \Cref{HP_model_dependence}, after enlarging $K$ and $S$, we may assume that there exists a blow-up map $\sigma\colon X\to \bP^2$ along at most $8$ points in general position of $\bP^2$ such that $\sigma$ is defined over $K$. If $\sigma(D)$ is reducible then the claim follows from \Cref{riducibile}. Otherwise, due to \Cref{caratterizzazione_complementi} and \Cref{HP_model_dependence}, after contracting some exceptional divisors and enlarging $K$ and $S$, we reduce to the case of \Cref{irreducible}.
\end{proof}
For the reader's convenience we collect some notation that we will use:
\begin{itemize}
    \item $D$ in an absolutely irreducible cubic $K$-curve in $\bP^2$ defined by an homogeneous cubic equation $F \in K[x,y,z]$,
    \item $P$ is a smooth $K$-point of $D$,
    \item $\sigma \colon X \to \bP^2$ is the blow-up at $P$,
    \item $\widehat{D}$ is the strict transform of $D$ under $\sigma$,
    \item $E$ is the exceptional divisor above $P$,
    \item $\lambda \colon X\to \bP^1$ is the pencil of strict transforms of lines through $P$,
    \item $\s \subset \bP^3$ is the cubic surface $w^3=F(x,y,z)$,
    \item $\rho \colon \s \to \bP^2$ is the projection from the point $[x:y:z:w]=[0:0:0:1]$ to the plane $H \coloneqq (w=0)$, that is, the map $[x:y:z:w] \mapsto [x:y:z]$.
\end{itemize}
\Cref{irreducible} does not follow from \Cref{fondamentalissimo} since $\lambda$ is the only fibration in conics with two points at infinity that we have on $X \setminus \widehat{D}$. Still the same strategy of \Cref{fondamentalissimo} applies: the idea is to employ the one-dimensional family of plane conics that Beukers used in \cite{beukers} to prove potential density of integral points for the complement of a cubic curve in $\bP^2$, for instance $\bP^2 \setminus D$. These conics are associated to a flex line $L$ of $D$, and they intersect $D$ in two points and $L$ in one point. Under certain hypotheses that are always met after enlarging $K$ and $S$, Beukers shows that there is an infinite subset of $p \in (L\setminus D)(\O_S)$ such that the Beukers' conic passing through $p$ is defined over $K$ and satisfies the hypotheses of \Cref{pell}. One then finds infinitely many Beukers' conics containing infinitely many $S$-integral points, thereby showing that $(\bP^2\setminus D)(\O_S)$ is Zariski dense.

There are two main obstructions to applying directly the argument of \Cref{fondamentalissimo}. The first issue is that the Beukers' conics do not form a linear system and so \Cref{generically_reducible} does not hold: indeed there are irreducible covers of $X$ for which the preimage of all but finitely many Beukers' conics is reducible \emph{over $K$}, so that these covers could lift all integral points on Beukers' conics. The second issue is that Beukers' construction of conics with infinitely many integral points gives only a \emph{thin} set of parameters, while we need a non-thin set.

Both obstructions are solved reinterpreting Beukers' conics in terms of certain conic fibrations on $\s$. More specifically, we will show that to any flex line of $D$ there correspond three coplanar lines contained in $\s$. For any line $\ell \subset \s$ we can consider the conic fibration $\mu_{\ell} \colon \s \dashrightarrow \bP^1$ associated to the linear system of planes containing the line. Then the conics constructed by Beukers are precisely the images under $\rho$ of the fibers of $\mu_{\ell}$.

The first issue is solved using that $\mu_{\ell}$ is a linear system: this allows to prove that, for any $\lambda$-unramified cover, the preimages of the Beukers' conics are reducible over $K$ only for a thin set of parameters. The second issue is solved by parametrizing the Beukers' conics associated to $\ell$ with another line $\ell' \subset \s$ which is skew with $\ell$ (after having suitably enlarged $K$ and $S$). In conclusion, one can apply the strategy of \Cref{fondamentalissimo} and prove \Cref{irreducible}. Let us now get into the details of the argument.

\subsection{Geometric preliminaries}
We first collect some facts on the geometry of irreducible plane cubic curves.
\begin{definition}
    A \emph{flex line} of a plane cubic is a line intersecting the cubic in a single point. A \emph{flex point} is a point of the cubic contained in a flex line.
\end{definition}
\begin{remark}
    If a flex $Q$ is a smooth point then the tangent line is the only flex line passing through $Q$. If a point $Q$ is singular then it is a flex, since its principal tangents are flex lines.
\end{remark}
Up to taking a linear transformation (defined over $\overline{K}$) we may assume that $D$ is given by a cubic equation of the form
\begin{equation*}
	F \, \colon \, x^3+zG(x,y,z)=0
\end{equation*}
where $G$ is a homogeneous degree $2$ polynomial. We will denote points in projective coordinates with the notation $[x:y:z]$.
\begin{lemma}\label{existence_of_flex}
    All irreducible plane cubics have at least two flex points over $\overline{K}$ and at least one of them is smooth.
\end{lemma}
\begin{proof}
    If $F=0$ is smooth then it has $9$ smooth flexes.
    If $F=0$ is singular, there exists a linear transformation (defined over $\overline{K}$) that sends it either to $x^3+z(x^2-y^2)=0$ or $x^3-zy^2=0$, and the claim easily follows. 
\end{proof}
Notice that the surface $\s$ is absolutely irreducible and is smooth when $D$ is smooth, and singular only at the points of $H$ corresponding to singular points of $D$. Moreover, notice that the map $\rho$ is a cover of degree $3$ which is étale over $\bP^2 \setminus D$ and totally ramified along $D$. 
\begin{remark}\label{corrispondenza_punti_infinito}
    In the following we will frequently identify $H$ with $\bP^2$ and $H \cap \s$ with $D$. In particular, for a curve $C \subset \s$, the points of $\rho(C) \cap D$ are in a one-to-one correspondence with $C \cap H$, so we will frequently identify them in the following.
\end{remark}
\begin{lemma}
    If $\ell$ is a line contained in $\s$ then $\rho(\ell)$ is a flex line of $D$.
\end{lemma}
\begin{proof}
    The map $\rho$ is the restriction to $\s$ of the projection $\bP^3 \dashrightarrow \bP^2$ from $[0:0:0:1]$ onto the plane $H$, which sends lines not containing $[0:0:0:1]$ to lines in $\bP^2$. Since $\s$ does not contain $[0:0:0:1]$, all lines in $\s$ are sent to lines in $H \cong \bP^2$ and so $\rho(\ell)$ is a line. 
    The plane $H$ does not contain lines of $\s$, so $\ell \cap H$ is a single point. By \Cref{corrispondenza_punti_infinito} also $\rho(\ell) \cap D$ is a single point and this proves that $\rho(\ell)$ is a flex line of $D$.
\end{proof}
In the opposite direction, we have the following:
\begin{lemma}\label{preimage_lines}
    Let $L \subset \bP^2$ be a flex line of $D$. Then $\rho^{-1}(L)$ is given by three coplanar concurrent lines and their common point lies on $H$.
\end{lemma}
\begin{proof}
    If $L(x,y,z)=0$ is a linear equation for the flex $L$, we have that $F=L\cdot Q+M^3$, where $Q$ is a homogeneous degree $2$ polynomial and $M$ is a linear form distinct from $L$. Then the points of $\rho^{-1}(L)$ satisfy $w^3=M^3$, and so the preimage of $L$ is given by the three lines $L=w-\epsilon^i M=0$ for $i=0,1,2$, where $\epsilon$ is a primitive third root of unity. These lines all lie in the plane $L(x,y,z)=0$ and meet in the point $w=L=M=0$.
\end{proof}
\begin{remark}\label{third_roots}
    It follows from the proof of \Cref{preimage_lines} that, if $K$ contains the third roots of unity, then for any $K$-rational flex line $L$ of $D$ also the three lines in $\rho^{-1}(L)$ are defined over $K$. 
\end{remark}

Any line $\ell$ in $\s$ defines a conic fibration $\mu_{\ell} \colon \s \dashrightarrow \bP^1$ associated to the linear system of planes containing $\ell$. The map $\mu_{\ell}$ is actually defined everywhere except on the point $\ell \cap H$ in case this is a singular point of $H$. We let $\widetilde{C}_u \coloneqq \overline{\mu_{\ell}^{-1}(u)}^{{\rm Zar}}$ for $u \in \bP^1$, which is a smooth conic for all but finitely many $u$. We define $C_u \coloneqq \rho(\widetilde{C}_u)$, which is again a smooth conic for all but finitely many $u \in \bP^1$.
\begin{remark}
    When $D$ has equation $x^3+zG=0$, the conics $C_u$ associated to the line $\ell \colon z=w-x=0$ are precisely the conics constructed by Beukers in \cite{beukers}. 
\end{remark}
\begin{lemma}
    The conic $\widetilde{C}_u$ meets $H$ in two points for all but finitely many $u$. If $\rho(\ell)\cap D$ is a smooth point of $D$ then the points of $\widetilde{C}_u \cap H$ vary algebraically in $u$. If $\rho(\ell)\cap D$ is a singular point of $D$ then $\widetilde{C}_u \cap H$ is given by $\ell \cap H$ and a point varying algebraically in $u$.
\end{lemma}
\begin{proof}
    Any conic $\widetilde{C}_u$ corresponds to a plane $\Lambda_u$ such that $\Lambda_u \cap \s=\ell \cup \widetilde{C}_u$. The line $L_u \coloneqq \Lambda_u \cap H$ passes through $q_{\ell} \coloneqq \ell \cap H$, which is the flex point associated to $L \coloneqq \rho(\ell)$. Following the identification of \Cref{corrispondenza_punti_infinito}, if $D$ is smooth then $L_u$ intersects $D$ (generically) in two more points varying algebraically with $u$. If $D$ is singular and $q_{\ell}$ is the singular point of $D$, then $L_u$ has a double intersection with $D$ in $q_{\ell}$ and intersects $D$ in another point varying algebraically in $u$. The claim follows from the fact that $\widetilde{C}_u\cap H=L_u \cap D$. 
\end{proof}
Using \Cref{corrispondenza_punti_infinito} we get the following:
\begin{cor}\label{punti_infinito_beukers}
    For all but finitely many points the conics $C_u$ intersect $D$ in two points corresponding uniquely to the points at infinity of $\widetilde{C}_u$. If $\rho(\ell)\cap D$ is a smooth point of $D$ then the points of $C_u\cap D$ vary algebraically in $u$. If $\rho(\ell)\cap D$ is a singular point of $D$ then $C_u \cap D$ is given by the flex point $\rho(\ell)\cap D$ and a point varying algebraically in $u$.  In particular, there is a one to one correspondence between the conics $\widetilde{C}_u$ in $\s$ and the conics $C_u$ in $\bP^2$.
\end{cor}

\begin{remark}\label{tripla_tangenza}
The conics $C_u$ are tangent with multiplicity $3$ to $D$ in both points of $C_u \cap D$. This follows from the fact that the preimage $\rho^{-1}(C_u)$ is given by three conics, one of which is $\widetilde{C}_u$ and the other two are obtained in the following way: intersect $\s$ with the two planes containing both $H\cap \widetilde{C}_u$ and one of the two lines in $\rho^{-1}(\rho(\ell))\setminus \ell$. Notice that these two other conics are members of the pencils of conics associated to the other two lines in $\rho^{-1}(\rho(\ell))$. 
\end{remark}
We want to find a line of $\s$ parametrizing the conics $\widetilde{C}_u$, and so the conics $C_u$ (due to \Cref{punti_infinito_beukers}).
\begin{lemma}
    Let $L$ and $L'$ be flex lines of $D$ associated to distinct flex points. Then each of the three lines in $\rho^ {-1}(L')$ intersects one and only one of the three lines in $\rho^{-1}(L)$.
\end{lemma}
\begin{proof}
    Clearly any line $\ell' \in \rho^ {-1}(L')$ intersects at least one line of $\rho^{-1}(L)$ (and viceversa). No line of $\rho^ {-1}(L')$ is contained in $\rho^{-1}(L)$, since $L$ and $L'$ are distinct: it follows that $\ell'$ intersects the plane $L=0$ (containing the lines of $\rho^{-1}(L)$) exactly in one point. If $\ell'$ intersected two lines $\ell_1$ and $\ell_2$ of $\rho^{-1}(L)$, this would imply that $\ell'$ intersects the plane $L=0$ in the flex point associated to $L$, which is the point at infinity where the three lines of $\rho^{-1}(L)$ meet. At the same time the point at infinity of $\ell'$ is the flex associated to $L'$: this would imply that $L$ and $L'$ are the same line, which contradicts the hypothesis. A dual argument proves that no two lines of $\rho^ {-1}(L')$ meet the same line of $\rho^{-1}(L)$, thereby proving the claim.
\end{proof}

\begin{cor}\label{section}
    Let $L'$ be a flex line distinct from $L$ and let $\ell'$ be one of the two lines in $\rho^{-1}(L')$ not intersecting $\ell$. Then $\ell'$ is a section of $\mu_{\ell}$. In other words, the points in $\ell'$ parametrize the conics $\widetilde{C}_u$ and so the conics $C_u$.
\end{cor}
\begin{proof}
    This follows from the fact that $\ell'$ intersects in one point the plane $\Lambda_u$ containing $\widetilde{C}_u$: since it does not intersect $\ell$, it must intersect $\widetilde{C}_u$ in one point.
\end{proof}

\begin{remark}
    We stress that while the conics $\widetilde{C}_u$ form a one dimensional linear system, this is not true for $C_u$, which is just an algebraic family of conics. 
\end{remark}
Since $P$ is smooth in $D$, by \Cref{punti_infinito_beukers} there is only one conic $C_u$ passing through $P$. We denote by $\widehat{C}_u$ the strict transform of $C_u$ under the blowup $\sigma \colon X \to \bP^1$, which is simply the preimage of $C_u$ if $C_u$ is not passing through $P$. The following result will play the same role of Hypothesis \ref{ipotesi_ramificazione} for \Cref{fondamentalissimo}.
\begin{lemma}\label{beuker_ramification}
	For all but finitely many $u$ the restriction $\lambda \colon \widehat{C}_u \to \bP^1$ is a finite map of degree $2$ ramified in two points not lying on $D$. Moreover the branch values are non-constant algebraic functions of $u$.
\end{lemma}
\begin{proof}
    Pick $u$ such that the conic $C_u$ is smooth and does not contain $P$, which happens for all but finitely many $u$. Since every line through $P$ intersects $C_u$ in two points and there are precisely two tangent lines to $C_u$ passing through $P$, we get that $\lambda_{|\widehat{C}_u}$ is a degree $2$ finite map ramified in two points.
    If one of the tangent lines meets $C_u$ on a point $Q \in D$, since $C_u$ is tangent to $D$ in $Q$ by \Cref{tripla_tangenza}, then the line would be tangent to $D$ in $Q$ (if $Q$ is a singular point, the line should be a principal tangent). There are only finitely many points on $D$ such that the lines joining them with $P$ are tangent to $D$ (recall that $P$ is a smooth point of $D$). Therefore, with the exception of the finitely many conics $C_u$ passing through these points, the tangents from $P$ to $C_u$ meet $C_u$ outside $D$.
	
	It remains to prove that there is no tangent line which is constant as $u$ varies.
	We limit ourselves to the case of $D$ smooth: the singular case can be dealt similarly. The statement is geometric, so we may work over $\overline{K}$ and assume that the smooth cubic $D$ is defined by a Weierstrass equation $y^2=x^3+ax+b$ and the line at infinity is the flex line $L$, meeting $D$ in $O$. There are $8$ more flexes given by the $3$-torsion points and they come in couples preserved by the reflection $\tau(y)= -y$. The conics $C_u$ are invariant under $\tau$: in particular, they meet $D$ generically in two points symmetric with respect to $\tau$.
	
	We first need to determine the singular elements of the family $C_u$. A conic $C_u$ is singular if and only if $\widetilde{C}_u$ is, and $\widetilde{C}_u$ is singular precisely when it is the product of two coplanar lines of $\s$ contained in a plane through $\ell$. One of these planes is $\rho^{-1}(L)$ and the corresponding conic $C_u$ is the double flex line $z^2=0$. There are four more planes intersecting $\s$ in two lines (in addition to $\ell$): their points of intersection with $H$ are couples of flexes $Q$ and $Q'$ with the same $x$-coordinate, that is, lying on the same vertical line. In this case $C_u$ is the product of the flex lines in $Q$ and $Q'$. Since there are always $8$ more flexes in addition to $O$, and they come in couples lying on the same vertical line, we find precisely four singular $C_u$'s (in addition to $z^2=0$). In particular, whenever $C_u$ meets $D$ in a point which is not a flex, then $C_u$ is smooth.
	
	Assume by contradiction that a branch point of the restriction $\lambda \colon \widehat{C}_u \to \bP^1$ is constant in $u$. Then there would be a line $M$ through $P$ which is tangent to all the conics $C_u$. Pick a couple of flexes $Q_1$ and $Q_2$ symmetric with respect to $\tau$. The singular conic $C_{u(Q_1)}$ passing through them is given by the product of the two flex lines in $Q_1$ and $Q_2$, meeting in a point $R$ outside $D$. The conic $C_{u(Q_1)}$ is invariant under $\tau$, so that $R$ must lie on the line $y=0$. Since $M$ intersects all smooth conics $C_u$ in one point, by continuity it must intersect also $C_{u(Q_1)}$ in one point, and so it must pass through $R$. Repeating the argument for all $4$ couples of flexes in vertical lines, we find that $M$ must be the line $y=0$. Using that $D$ is in Weierstrass form, $y=0$ intersects $D$ in three distinct points, corresponding to the three non-trivial $2$-torsion points. Since the flexes are precisely the $3$-torsion points, the non-trivial $2$-torsion points are not flexes. In particular, the conics $C_u$ passing through the non-trivial $2$-torsion points are smooth. For any non-trivial $2$-torsion point $Q$ the conic $C_u(Q)$ passing through it is tangent in $Q$ to both $M$ and $D$. Therefore $M$ is tangent to $D$ in all the three non-trivial $2$-torsion points, thereby contradicting Bezout's Theorem.
 \end{proof}

\begin{cor}\label{solo_finiti_liftano_beukers}
    Let $T \subset \bP^1(K)$ be a thin set. Then for all but finitely many $u \in \bP^1(K)$ the set $\lambda(\widehat{C}_u (\O_S)) \cap T$ is finite.
\end{cor}
\begin{proof}
    This follows with the same strategy of \Cref{solo_finiti_liftano}.
\end{proof}

Let $\rho_X \colon \s_X \to X$ be the base change of $\rho \colon \s \to \bP^2$ to $X$. The surface $\s_X$ is birational to $\s$ and so it is absolutely irreducible. Then the map $\rho_X$ is finite of degree $3$. Since $\rho$ is totally ramified above $D$ and the blown up point $P$ lies on $D$, the map $\rho_X$ is also totally ramified above $E$.
Let $Y$ be a normal absolutely irreducible projective variety together with a finite map $\pi \colon Y \to X$ of degree $>1$. Then the base change $\rho_Y \colon\s_Y \to Y$ of $\rho \colon \s \to \bP^2$ to $Y$ fits into the following commutative diagram
\begin{equation}\label{diagramma_s_y}
	\begin{tikzcd}
		\s_{Y} \arrow{r}{\pi'} \arrow{d}{\rho_Y} & \s_X \arrow{r}{\sigma'}\arrow{d}{\rho_X} & \s \arrow{d}{\rho} \\
		Y \arrow{r}{\pi} & X \arrow{r}{\sigma} & \bP^2
	\end{tikzcd}
\end{equation}
In the above diagram all the vertical maps are finite, since $\rho$ is finite. We have the following:
\begin{lemma}\label{s_abs_irred}
	If $\pi$ is unramified above the exceptional divisor $E$, then $\s_Y$ is absolutely irreducible.
\end{lemma}
\begin{proof}	
	Assume by contradiction that $\s_Y$ is reducible over $\overline{K}$. The restriction of $\rho_Y$ to any irreducible component of $Y$ is still a finite map. Since $\rho_Y$ has degree $3$, there must be an irreducible component $Z$ of $\s_Y$ such that the restriction $\rho_{Y \vert Z} \colon Z \to Y$ is a degree $1$ finite map. Using that $Y$ is normal, \cite[Corollary $4.6.6$]{liu} ensures that $\rho_{Y \vert Z}$ is an isomorphism. In particular, there exists a finite morphism $Y \to \s_X$ (since $\s_Y \to \s_X$ is finite) such that the following diagram commutes
	\begin{equation*}
		\begin{tikzcd}
			& \s_X \arrow{d} \\
			Y \arrow{r}{\pi} \arrow{ur} & X 
		\end{tikzcd}
	\end{equation*} 
	Since $\s_X \to X$ is totally ramified above $E$, then $Y \to X$ is also ramified over $E$, contradicting the assumption.
\end{proof}

\subsection{Integral points on $X$}
We will now describe how to extend $K$ and $S$ and how to choose the integral models for $(\bP^2,D)$, $(\s, \s \cap H)$ and $(X, \widehat{D})$ that we will use in the proof of \Cref{irreducible}. A summary of the necessary enlargement of $K$ and $S$ is given at the end of this section.

There exists a linear automorphism of $\bP^2$, possibly defined over $\overline{K}$, which sends $D$ to the zero locus of a homogenous cubic polynomial $F \in \overline{K}[x,y,z]$ of the form $F=x^3+zG(x,y,z)$. Enlarging $K$ and $S$, we may assume that $F \in \O_S[x,y,z]$. We let $\mathcal{D}\coloneqq (F=0)$. It follows that, up to enlarging $K$ and $S$, $(\bP^2_{\O_S},\mathcal{D})$ is an $\O_S$-integral model for $\bP^2\setminus D$.
The blown-up point $P \in D$ corresponds to a $\O_S$-point $\mathcal{P}$ contained $\mathcal{D}$. Let $\mathcal{X}$ be the blow-up of $\bP^2_{\O_S}$ along $\mathcal{P}$ and let $\widehat{\mathcal{D}}$ be the strict transform of $\mathcal{D}$. Then, up to enlarging $S$, we have that $(\mathcal{X},\widehat{\mathcal{D}})$ is an $\O_S$-integral model for $(X, \widehat{D})$.
For the pair $(\s, H)$ we take the integral model $\Proj\left(\O_S[x,y,z,w]/(w^3-F) \right)$ and the hyperplane divisor cut out by $(w=0)$.
Since all the integral models have been fixed, from now on we will follow the notational convention of \Cref{notation_clarification}.

Since $\rho \colon \s \to \bP^2$ extends to an $\O_S$-morphism of pairs between the integral models of $(\s,H)$ and $(\bP^2,D)$, we have the following:
\begin{lemma}\label{integral_rho}
    With the above choice of integral models we have that $\rho$ maps $(\s\setminus H)(\O_S)$ to $(\bP^2\setminus D)(\O_S)$.
\end{lemma}
\begin{remark}\label{integral_blowup}
Notice that the same does \emph{not} hold for the blow-up map $\sigma \colon X \to \bP^2$. The reason is that $\sigma^{-1}(D)=E+\widehat{D}$, so $\sigma$ maps a point $q \in X(\O_S)$ to $(\bP^2\setminus D)(\O_S)$ only if $q$ is also $S$-integral with respect to $E$. However, if $p \in (\bP^2\setminus D)(\O_S)$ and there exists $q \in X(K)$ such that $\sigma(q)=p$, then $q \in X(\O_S)$ and $q$ is actually $S$-integral also with respect to $E$.
\end{remark}

By \Cref{existence_of_flex} we know that there is at least one flex line $L$ associated to a smooth point of $D$ (possibly the blown-up point $P$): we enlarge $K$ so that $L$ is defined over the base field. We further enlarge $K$ to ensure that another flex line $L'$ (possibly associated to a singular flex) is defined over the base field. We enlarge $K$ and $S$ so that $K$ contains the third roots of unity and $\# S \ge 2$ (in particular, the hypotheses of \Cref{punti_infiniti_giusto} are met). Using \Cref{preimage_lines}, \Cref{third_roots} and \Cref{section} we get that the cubic surface $\s$ contains two skew $K$-lines $\ell$ and $\ell'$ such that $\rho(\ell)=L$, $\rho(\ell')=L'$ and $\ell'$ is a section of the conic fibration $\mu_l \colon \s \dashrightarrow \bP^1$ associated to $\ell$. Enlarging $S$ we may assume that $\ell'(\O_S)$ is non-empty and hence non-thin by \Cref{pell}. We can finally obtain the first non-thin family of conics with infinitely many integral points:
\begin{proposition}\label{non_thin_beukers_family}
    With the above enlargements of $K$ and $S$, the set $\{u \in \bP^1(K) \, \vert \, \widehat{C}_u(\O_S) \, \text{is infinite}  \}$ is not thin.
\end{proposition}
\begin{proof}
    Let $\mathcal{U} \coloneqq \mu_{\ell}(\ell'(\O_S)) \setminus \{ u \, \vert \, \text{$\widetilde{C}_u$ is singular}\}$. Since $\ell'(\O_S)$ is not thin and $\ell'$ is a section of $\mu_l$, then also $\mathcal{U}$ is not thin. For every $u \in \mathcal{U}$ the conic $\widetilde{C}_u$ contains one $S$-integral point and so by \Cref{punti_infiniti_giusto} it contains infinitely many. \Cref{integral_rho} ensures that $C_u(\O_S)$ is infinite for $u \in \mathcal{U}$. \Cref{integral_blowup} gives the claim.
\end{proof}

The second family of conics is much more easily produced. We let $\widehat{L}_l \coloneqq \lambda^{-1}(l)$, which is the strict transform under $\sigma$ of a line through $P$. Notice that the conics $\widehat{L}_l$ generically intersect $\widehat{D}$ in two points. The exceptional divisor $E$ over $P$ is a smooth rational curve defined over $K$ intersecting $\widehat{D}$ in a single point: in particular, possibly after enlarging $S$, we may assume that $(E\setminus \widehat{D})(\O_S)$ is non-empty, hence non-thin by \Cref{pell}. 
\begin{proposition}\label{non_thin_lambda}
     With the above enlargements of $K$ and $S$, the set $\{l \in \bP^1(K) \, \vert \, \widehat{L}_l(\O_S) \, \text{is infinite}  \}$ is non-thin.
\end{proposition}
\begin{proof}
    The claim follows from \Cref{punti_infiniti_giusto} and the fact that $E$ is a section of $\lambda$.
\end{proof}

To summarize, we have enlarged $K$ and $S$ so that the following hold:
\begin{itemize}
    \item there exists $\mathcal{P}\in \bP^2_{\O_S}(\O_S)$ and a homogeneous cubic polynomial $F\in \O_S[x,y,z]$ of the form $F=x^3+zG(x,y,z)$ such that $\mathcal{P}$ lies on $\mathcal{D}\coloneqq (F=0)$ and such that an $\O_S$-integral model for the pair $(X,\widehat{D})$ is given by $(\mathcal{X},\widehat{\mathcal{D}})$, where $\mathcal{X}$ is the blow-up of $\bP^2_{\O_S}$ along $\mathcal{P}$ and $\widehat{\mathcal{D}}$ is the strict transform of $\mathcal{D}$;
    \item two flex lines $L$ and $L'$ of $(F=0)$ are defined over $K$ and the flex at $L$ is a smooth $K$-point of $(F=0)$;
    \item $K$ contains the third roots of unity and $\# S \ge 2$;
    \item $(\ell'\setminus H)(\O_S)$ and $(E\setminus \widehat{D})(\O_S)$ are both non-empty.
\end{itemize}

\subsection{Proof of \Cref{irreducible}}

We can finally proceed to the proof of the Hilbert Property.

\begin{proof}
    Assume by contradiction that $(X\setminus D)(\O_S)$ is thin. By \Cref{reduction_finite} there exists finitely many absolutely irreducible normal projective $K$-varieties $Y_i$ equipped with finite $K$-maps $\pi_i \colon Y_i \to X$ of degree $>1$ such that $(X\setminus D)(\O_S) \setminus \bigcup_i(\pi_i(Y_i(K)))$ is contained in a curve. 
    Arguing as in the proof of \Cref{fondamentalissimo}, we can find a thin set $T \subset \bP^1(K)$ such that, for $l \notin T$, only finitely many points of $\widehat{L}_l(\O_S)$ lift to the covers $\pi_i$ which are $\lambda$-ramified. Therefore, if all covers $\pi_i$ were $\lambda$-ramified, a contradiction would follow from \Cref{non_thin_lambda}. We may then assume that there is at least one $\lambda$-unramified cover. Arguing as in the proof of \Cref{fondamentalissimo}, we can find a thin set $T \subset \bP^1(K)$ (obtained enlarging the previous $T$) such that, for $l \notin T$, whenever $\widehat{L}_l(\O_S)$ is infinite then the $\lambda$-unramified covers lift \emph{all} $S$-integral points on $\widehat{L}_l$. We will prove the following:

	\begin{claim}\label{claim3}
         There exists a thin set $T' \subset \bP^1(K)$ such that, for $u \notin T'$, only finitely many $S$-integral points of $\widehat{C}_u$ lift to $\lambda$-unramified covers.
    \end{claim}
	
	We will examine each $\lambda$-unramified cover. Let us first deal with the case in which $\pi_i^{-1}(\widehat{C}_u)$ is reducible over $\overline K$ for all but finitely many $u$. Then, in the notation of diagram \eqref{diagramma_s_y}, the same would be true for $(\sigma' \circ \pi_i')^{-1}(\widetilde{C}_u)$, since this is the fiber product of $\rho \colon \widetilde{C}_u \to C_u$, which is an isomorphism, and $\sigma \circ \pi_i \colon \pi_i^{-1}(\widehat{C}_u) \to C_u$, which is a reducible cover over $\overline{K}$. The cover $\pi_i$ is not ramified over $E$ since it is a $\lambda$-unramified cover, and so $\s_{Y_i}$ is absolutely irreducible by \Cref{s_abs_irred}. We may then apply \Cref{generically_reducible} to the map $\sigma' \circ \pi_i' \colon \s_{Y_i} \to \s$ to find a thin set $T_i'$ such that, for $u \notin T_i'$, the curve $(\sigma' \circ \pi_i')^{-1}(\widetilde{C}_u)$ is irreducible over $K$. Then also $\pi_i^{-1}(\widehat{C}_u)$ is irreducible over $K$ for $u \notin T_i'$. Hence, for $u \notin T_i'$, $\pi_i^{-1}(\widehat{C}_u)$ is irreducible over $K$ but reducible over $\overline{K}$, which implies that it contains finitely many $K$-rational points by \Cref{remark_serre}. In particular, for $u \notin T_i'$, only finitely many $S$-integral points of $C_u$ lift to $Y_i(K)$.
	
	Let us assume that $\pi_i^{-1}(\widehat{C}_u)$ is irreducible over $\overline K$ for all but finitely many $u$. Then, since $X\setminus \widehat{D}$ is simply connected by \Cref{caratterizzazione_complementi}, $\pi_i$ is ramified along at least one curve $\widehat{L}_l$. We can apply \Cref{beuker_ramification} to conclude that, for all but finitely many $u$, the smooth conic $\widehat{C}_u$ intersects transversally the branch divisor in at least one point in its affine part. We let $T_i'$ be the finite set of exceptions. Arguing as we did for \Cref{claim1} in the proof of \Cref{fondamentalissimo}, \Cref{trick_siegel} allows us to conclude that only finitely many points of $\widehat{C}_u(\O_S)$ lift to $Y_i(K)$. \Cref{claim3} follows taking $T'=\bigcup_i T'_i$.

	Using \Cref{non_thin_beukers_family} and \Cref{solo_finiti_liftano_beukers} we can pick an $u \notin T'$ such that $\widehat{C}_u(\O_S)$ is infinite and for all but finitely many $p \in \widehat{C}_u(\O_S)$ we have $\lambda(p) \notin T$. For any such $S$-integral point $p$ the curve $\widehat{L}_{\lambda(p)}$ contains $p$ and so has infinitely many $S$-integral points by \Cref{punti_infiniti_giusto}. Since $\lambda(p)\notin T$ we have that \emph{all} of $\widehat{L}_{\lambda(p)}(\O_S)$ lifts to $\lambda$-unramified covers, including $p$. Then all but finitely many points of $\widehat{C}_u(\O_S)$ would lift to $\lambda$-unramified covers, contradicting \Cref{claim3}.
 \end{proof}

\subsection{A non-potential example}\label{esempio_non_potenziale}

Whenever the conditions listed before the proof of \Cref{irreducible} hold over a given ring of integers $\O_{K,S}$, then the IHP holds over $\O_{K,S}$ without the need to perform any enlargement. However these conditions cannot hold over specific ring of integers, for instance over $\bZ$, because of the third condition. Still in some cases one can prove the IHP over $\bZ$: for example one can show that if $D \coloneqq (zy^2-x^3-z^3=0)$ and $X$ is the blowup of $\bP^2$ at $[x:y:z]=[2:3:1]$, then $(X\setminus \widehat{D})(\bZ)$ is not thin. Here the integral model for $X$ is the blowup of $\bP^2_{\bZ}$ at the section over $\Spec \bZ$ defined by $[2:3:1]$; the integral model for $\widehat{D}$ is the strict transform of the Zariski closure of $D$ in $\bP^2_{\bZ}$.

The essential difficulty lies in the fact that we cannot use \Cref{punti_infiniti_giusto} to show that a conic has infinitely many integral points. More explictly, due to \Cref{pell}, a smooth conic with one integral point contains infinitely many if and only if its divisor at infinity consists of one or two real points. We will only sketch how to circumvent these issues: in a forthcoming work \cite{coccia_sarnak} we will explore these ideas in detail and will find general conditions under which the set of $\bZ$-points on an affine smooth cubic surface is not thin.

The first step is to construct a non-thin set of (parameters for) Beukers' conics containing infinitely many integral points. There are precisely three lines defined over $\bQ$ in the surface $\s \coloneqq (w^3=zy^2-x^3-z^3)$, all contained in the plane $w+x=0$. Their equations are
\begin{equation*}
    L_1 \colon z=w+x=0 \qquad L_2 \colon z-y=w+x=0 \qquad L_3 \colon z+y=w+x=0
\end{equation*}
and they correspond to the flexes $[0:1:0], [0:1:1]$ and $[0:-1:1]$ of $D$, respectively. Notice that all these lines contain a non-thin set of integral points. Let $\mu_i \colon \s \to \bP^1$ be the conic fibration associated to $L_i$. Since all the $\bQ$-lines are coplanar, we cannot use $L_2$ or $L_3$ to produce a non-thin set of fibers of $\mu_1$ containing infinitely many integral points as we did in \Cref{non_thin_beukers_family}. To circumvent this issue we first check that for all but finitely many $p \in L_1(\bZ)$ the conic $\mu_1^{-1}(\mu_1(p))$ is smooth, has two real points at infinity and contains one integral point (namely $p$). We can then apply \Cref{pell} and obtain that $\mu_1^{-1}(\mu_1(p))(\bZ)$ is infinite. However these conics are parametrized by the thin set $\mu_1(L_1(\bZ))$: this is because the restriction $\mu_1 \colon L_1 \to \bP^1$ is a degree $2$ map. Still, looking at the fibers of $\mu_2$ passing through integral points of the conics $\mu_1^{-1}(\mu_1(p))$ \emph{while $p$ varies}, one can show that there is a non-thin set $\mathcal{U}$ such that for $u \in \mathcal{U}$ the conic $\mu_2^{-1}(u)$ contains infinitely many integral points. The strategy involves the study of the fiber product $W_p$ of $\mu_2 \colon D \to \bP^1$ and $\mu_2 \colon \mu_1^{-1}(\mu_1(p)) \to \bP^1$ similarly to \Cref{solo_finiti_liftano} and \Cref{just_need_one}, with the extra complications coming from checking that the hypotheses of \Cref{pell} are met. Indeed one needs to show that for all but finitely many $x\in \mu_1^{-1}(\mu_1(p))(\bZ)$ the preimages of $x$ under the base change map $f_p \colon W_p \to \mu_1^{-1}(\mu_1(p))$ are real and non-rational. That $f_p^{-1}(x)$ is given by real points comes from the fact that, for all but finitely many $p$, $f_p^{-1}(\mu_1^{-1}(\mu_1(p))\cap D)$ consists of smooth real points where $f_p$ is unramified, so if $x$ is sufficiently close to $\mu_1^{-1}(\mu_1(p))\cap D$ then it lifts to real points. Moreover, using \Cref{trick_siegel}, one can show that $f_p^{-1}(x)$ is given by \emph{non}-rational points for all but finitely many $x$.
To sum up, this allows to produce the sought non-thin family of Beukers' conics with infinitely many integral points.

It remains to show that for all but finitely many $u \in \mathcal{U}$ the corresponding Beukers' conic $\widehat{C}_u$ contains infinitely many integral points $p$ such that the line $\widehat{L}_{\lambda(p)}$ has two real points at infinity and so contains infinitely many integral points. Again, this is accomplished studying the fiber product of $\lambda \colon D \to \bP^1$ and $\lambda \colon \widehat{C}_u \to \bP^1$, using similar arguments as the above. The IHP then follows as in the proof of \Cref{irreducible}.